\documentclass[leqno,12pt]{amsart} 
\usepackage{amsmath, amssymb}
\usepackage{amsthm} 
\usepackage[all]{xy}
\usepackage{enumerate}
\renewcommand{\thefootnote}{} 
\theoremstyle{plain} 
\newtheorem{theorem}{\indent\sc Theorem}[section]
\newtheorem{lemma}[theorem]{\indent\sc Lemma}
\newtheorem{corollary}[theorem]{\indent\sc Corollary}
\newtheorem{proposition}[theorem]{\indent\sc Proposition}

\theoremstyle{definition} 
\newtheorem{definition}[theorem]{\indent\sc Definition}
\newtheorem{remark}[theorem]{\indent\sc Remark}
\newtheorem{example}[theorem]{\indent\sc Example}
\newtheorem{notation}[theorem]{\indent\sc Notation}

\newcommand{\C}{\mathbb{C}}
\newcommand{\R}{\mathbb{R}}
\newcommand{\Z}{\mathbb{Z}}
\newcommand{\N}{\mathbb{N}}

\newcommand{\Pen}{\mathrm{Pen}}

\def\Im{\mathop{\mathrm{Im}}\nolimits}

\newcommand{\abs}[1]{\lvert#1\rvert}

\newcommand{\norm}[1]{\lVert#1\rVert}

\newcommand{\limone}{\varprojlim{\!}^1}

\newcommand{\U}{\mathcal{U}}
\newcommand{\V}{\mathcal{V}}

\newcommand{\X}{\mathcal{X}}
\newcommand{\Y}{\mathcal{Y}}
\newcommand{\cZ}{\mathcal{Z}}
\newcommand{\Horo}{\mathcal{H}}

\newcommand{\EG}{\underline{E}G}
\newcommand{\EP}{\underline{E}P}
\newcommand{\Xaug}{X(G,\mathbb{P},\mathcal{S})}

\newcommand{\EX}{EX(G,\mathbb{P})}

\newcommand{\famP}{\mathbb{P}}

\newcommand{\vect}[1]{\mathbf{#1}}

\newcommand{\grad}{\mathbf{d}}

\newcommand{\dX}{\partial X}

\newcommand{\rK}{\tilde{K}}
\newcommand{\OW}{\mathcal{O}W}
\newcommand{\LieG}{\mathbf G}

\newcommand{\prodG}{\mathbb{G}}

\newcommand{\EGG}{\hat{E}\mathbb{G}}
\makeatletter
\def\address#1#2{\begingroup
\noindent\parbox[t]{7.8cm}{%
\small{\scshape\ignorespaces#1}\par\vskip1ex
\noindent\small{\itshape E-mail address}%
\/: #2\par\vskip4ex}\hfill%
\endgroup}%
\makeatother
\title{\uppercase{Coronae of product spaces and the Coarse Baum-Connes conjecture}}
\author{
\textsc{Tomohiro Fukaya, Shin-ichi Oguni} 
}
\date{} 
\begin{document}

\maketitle

\footnote{ 
T.Fukaya and S.Oguni were supported by Grant-in-Aid for Scientific
Researches for Young Scientists (B) (23740049), (24740045) 
respectively from Japan Society of Promotion of Science.
}
\renewcommand{\thefootnote}{\fnsymbol{footnote}} 

\begin{abstract}
We study the coarse Baum-Connes conjecture for product spaces and
product groups. We show that a product of CAT(0) groups, polycyclic
groups and relatively hyperbolic groups which satisfy some
assumptions on peripheral subgroups, satisfies the coarse Baum-Connes
conjecture.  For this purpose, we construct and analyze an appropriate
compactification and its boundary, ``corona'', of a product of proper
metric spaces.
\end{abstract}

\section{Introduction}
\label{sec:introduction}
Let $X$ be a proper metric space. 
Roe \cite{MR1147350} constructed the following coarse assembly map
\begin{align*}
 \mu_*(X)\colon KX_*(X) \to K_*(C^*(X)),
\end{align*}
where the left hand side is the coarse $K$-homology of $X$, and the right
hand side is the $K$-theory of the Roe algebra of $X$.  The coarse
Baum-Connes conjecture is a conjecture which states that for a proper
metric space of bounded geometry, the coarse assembly map $\mu_*(X)$ is
an isomorphism. This conjecture is deeply related to the differential
topology.  See \cite{MR1147350}, \cite{MR1388312}, \cite{MR1344138} and
\cite{MR1817560}. We remark that expander graphs with large girth do not
hold the conjecture (see \cite{Higson-counter-exmples},
\cite{HLS-couter-example-BC} and also \cite{Willett-Yu-higher-indexI}),
but the conjecture for finitely generated groups is still open. In this
paper, we study the coarse Baum-Connes conjecture for product spaces and
product groups by using coronae.

Throughout this paper we equip every finitely generated group with a
left invariant proper metric like a word metric with respect to a finite
generating set.  We remark that all left invariant proper metrics on the
group are mutually coarsely equivalent. We also equip a product of
metric spaces with the $l^1$ metric, which is coarsely equivalent to the
$l^2$ metric.

We ask the following question: Let $G_1,\dots,G_n$ be finitely generated groups.
If the coarse assembly map $\mu_*(G_i)$ is an isomorphism for $1\leq i\leq n$, 
then is the coarse assembly map $\mu_*(\prod_{i=1}^n G_i)$ an isomorphism?

The question is not interesting if every $G_i$ admits a coarse embedding
into a Hilbert space, because in this case, so does the product
$\prod_{i=1}^n G_i$, and thus $\mu_*(\prod_{i=1}^n G_i)$ is an
isomorphism by Yu's result~\cite[Theorem 1.1]{MR1728880}.  However, it
is unknown whether every CAT(0) group admits a coarse embedding into a
Hilbert space, nevertheless, it is known that every CAT(0) group
satisfies the coarse Baum-Connes conjecture \cite{MR1388312},
\cite{WillettThesis}, \cite[Theorem 1.1 and Remark 1.2]{Busemann_cBC}.　
Therefore, it is worthwhile to study this question in the case where one
of $G_i$ is a CAT(0)-group.
The main result of this paper is the following.
\begin{theorem}
\label{th:main-thoerem} For $1\leq i\leq l$, let $G_i$ be a finitely
 generated group. We suppose that each $G_i$ is a hyperbolic group, a
 CAT(0)-group, or a polycyclic group.

For $1\leq j\leq m$, let $G^j$ be a finitely generated group which is
hyperbolic relative to $\famP^j= \{P^j_1,\dots,P^j_{k^j}\}$, where
$P^j_{r}$ is an infinite finitely generated subgroup of $G^j$ with
infinite index. We assume that $P^j_{r}$ is a product
of some of hyperbolic groups, CAT(0)-groups, and polycyclic groups. We
also assume that each $P^j_{r}$ admits a finite $P^j_{r}$-simplicial
complex which is a universal space for proper actions.

Then the coarse Baum-Connes conjecture for the product group
$\prod_{j=1}^m G^j \times \prod_{i=1}^l G_i$ holds.
\end{theorem}

To obtain the above result, we need to know how the coarse assembly map
behave under the finite product of CAT(0)-spaces and geodesic Gromov
hyperbolic spaces which are not necessarily coarsely equivalent to
finitely generated groups. In order to study such a product space, 
we construct and analyze an appropriate compactification 
and its boundary, ``corona'', of a product of proper metric spaces.
The first half of this paper is devoted to this.
In Section~\ref{sec:join-corona}, we prove the following.
\begin{proposition}
\label{prop:product-corona-join}
Let $\{X_i\}_{i=1}^n$ be a finite sequence of proper metric spaces.
Let $\dX_i$ be a corona of $X_i$. 
Then the join $\dX_1\star \dots \star \dX_n$ is a corona of 
 $X_1\times \dots \times X_n$.
\end{proposition}

The following is a key to the proof of
Theorem~\ref{th:main-thoerem}.
\begin{theorem}
\label{th:corona-assembly}
Let $\{(X_i,o_i,W_i)\}_{i=1}^n$ be a finite sequence of proper metric spaces, 
base points, and compact metrizable spaces.
We suppose that there exist integers $0\leq k\leq l \leq m\leq n$ such that 
\begin{itemize}
 \item For $1\leq i\leq k$, 
       $X_i$ is a geodesic Gromov hyperbolic space, $o_i\in X_i$ and $W_i$ is the Gromov boundary,
 \item For $k< i\leq l$, 
       $X_i$ is an open cone over a compact metrizable space $W_i$ with the apex $o_i$,
 \item For $l< i\leq m$, 
       $X_i$ is a Busemann space, $o_i\in X_i$, and $W_i$ is the visual boundary.
 \item For $m< i\leq n$, 
       $X_i$ is a $p_i$-dimensional simply connected solvable Lie group with a lattice, 
       $o_i$ is the unit of $X_i$, and $W_i= S^{{p_i}-1}$ is a corona as 
       in Lemma~\ref{lem:homeo-coarse-map}.
\end{itemize}

Then the coarse assembly map and the transgression map 
\begin{align*}
 &\mu_{\prod X_i}\colon KX_*(X_1\times \dots \times X_n) 
\to K_{*}(C^*(X_1\times \dots \times X_n)) \\
 &T_{\star W_i}:KX_*(X_1\times \dots \times X_n) \to \rK_{*-1}(W_1\star \dots \star W_n)
\end{align*}
are isomorphisms.
Moreover the following map is an isomorphism.
\begin{align*}
b_{\star W_i}\colon K_*(C^*(X_1\times \dots X_n))
\to \rK_{*-1}(W_1 \star \dots \star W_n).
\end{align*}
Here $b_{\star W_i}$ is a homomorphism which is constructed
in \cite[Appendix]{MR1388312}.
\end{theorem}

We need to discuss more to show Theorem~\ref{th:main-thoerem}.
Details are given in Section 10.
On the other hand, the case where relatively hyperbolic groups do not
appear directly follows from
Theorem~\ref{th:corona-assembly} and also we have the following, which
means that
we can compute the $K$-theory of the Roe algebra by the reduced
$K$-homology of the corona.

\begin{corollary}
\label{cor:compute-K-theory-by-bdry}
\label{cor:product-of-groups}
For $1\leq i\leq n$, let $G_i$ be a finitely generated group, and let
$\partial G_i$ be a corona of $G_i$. We suppose that each of $(G_i,\partial G_i)$ is 
one of the following.
\begin{itemize}
 \item $G_i$ is a hyperbolic group and $\partial G_i$ is the Gromov boundary,
 \item $G_i$ is a CAT(0)-group and $\partial G_i$ is the visual boundary 
       defined in Section~\ref{sec:busemann-def}, or
 \item $G_i$ is a polycyclic group which is commensurable to 
a lattice of a $p_i$-dimensional simply connected solvable Lie group and 
$\partial G_i = S^{p_i-1}$ as in Lemma~\ref{lem:homeo-coarse-map}.
\end{itemize}
Then the following map is an isomorphism.
\begin{align*}
b_{\star \partial G_i}\colon K_*(C^*(G_1\times \dots  \times G_n)) 
 \to \rK_{*-1}(\partial G_1 \star \dots \star \partial G_n).
\end{align*}
\end{corollary}

We remark that every polycyclic group $G$ admits a normal subgroup $G'$
of finite index in $G$ which is isomorphic to a lattice in a simply
connected solvable Lie group. See \cite[Theorem 4.28.]{MR0507234}


The organization of this paper is as follows.  In
Section~\ref{sec:comp-at-infin}, we review coarse compactifications and
coronae. In Section~\ref{sec:corona-product-space}, we construct a
corona of a product of proper metric spaces by using the join of
topological spaces.  In Section~\ref{sec:corona-busem-spac}, we review
the visual boundary for Busemann spaces and a natural compactification
for open cones.  In Section~\ref{sec:contr-comp-at}, we show that in the
setting of Theorem~\ref{th:corona-assembly} without hyperbolic spaces,
the compactification of the product space is contractible. 
This property is crucial for our proof of 
Theorem~\ref{th:corona-assembly}.
In Section~\ref{sec:revi-coarse-algebr}, we review the coarse algebraic
topology. In Section~\ref{sec:simply-conn-solv}, we study the coarse
geometric/topological property of simply connected solvable Lie groups
with lattices. 
In Section~\ref{sec:proof-theor-refth:c}, we give a proof of
Theorem~\ref{th:corona-assembly} and
Corollary~\ref{cor:compute-K-theory-by-bdry}.  In
Section~\ref{sec:augmented-space}, we review relatively hyperbolic
groups. In Section~\ref{sec:proof-of-main-th}, we give a proof of
Theorem~\ref{th:main-thoerem} based on the arguments in \cite{relhypgrp}.


\section{Compactification at infinity}
\label{sec:comp-at-infin}
\subsection{Compactification at infinity}
For a technical reason, we need to consider ``a compactification at
infinity'' of a metric space which is not proper. Therefore, we
introduce the following notion.

\begin{definition}
 Let $X$ be a metric space. 
A {\itshape compactification at infinity} of $X$ is a Hausdorff space
$\overline{X}$ with an embedding $i\colon X\to \overline{X}$ such that $i(X)$ 
is an open subset, every closed bounded subset of $X$ is also closed
in $\overline{X}$ and 
the complement $\overline{X}\setminus i(X)$ is compact. We identify $X$ with $i(X)$.
\end{definition}

\begin{example}
Let $X$ be a proper metric space. Then a compactification of $X$ in the usual sense is 
a compactification at infinity. The converse is not true even if 
$i(X)$ is dense in $\overline{X}$. Indeed if we consider $\R$ with the
euclidean metric, $(-1,1]$ as $\overline{\R}$
and $i:\R\ni r\mapsto r/(1+|r|)\in (-1,1]$, then $(-1,1]$ is a
compactification of $\R$ at infinity,
but not a compactification of $\R$.

Let $W$ be a compact set in the unit sphere of a separable Hilbert space $l^2$.
Let $\OW$ be an open cone of $W$. 
For $d>0$, we denote by $\Pen(\OW,d)$ the closed $d$-neighborhood of $\OW$ in $l^2$.
Then a Hausdorff space $\overline{\Pen(\OW,d)}$, defined in Section~\ref{sec:cone-def}, is a 
compactification of $\Pen(\OW,d)$ at infinity. If $l^2$ is infinite dimensional, 
then $\Pen(\OW,d)$ is not proper, thus $\overline{\Pen(\OW,d)}$ is not compact. 
See Section~\ref{sec:cone-def} for details. 
We use this example in the proof of Proposition~\ref{prop:coarsening}
\end{example}

\subsection{Corona and coarse compactification}
\label{sec:higs-comp}
Here we recall the definitions of the Higson compactification and coarse
compactifications for a proper metric space. 
For details, see \cite[Section 5.1, 5.2]{MR1147350} and 
\cite[Section 2.3]{MR2007488}.
\begin{definition}
\label{def:Higson_function}
 Let $X$ be a proper metric space. Let  $f\colon X\to \C$ be a
 bounded continuous function. We denote by $\grad f$ the function
\begin{align}
\label{eq:def-of-df}
 \grad f(x,y) = f(y) - f(x) \colon X\times X \to \C.
\end{align}  
We say that $f$ is a {\itshape Higson function}, or, 
of {\itshape vanishing variation}, 
if for any $R>0$ and any $\epsilon> 0$, there exists a bounded set 
$K$ such that for any $(x,y)\in X \times X\setminus K\times K$, 
if $d(x,y)\leq R$, then $\abs{\grad f(x,y)}<\epsilon$.
\end{definition}

The space of Higson functions on a proper metric space $X$ forms a unital
$C^*$-subalgebra of bounded continuous functions on $X$, which we
denote $C_h(X)$. By the Gelfand-Naimark theory, $C_h(X)$ is isomorphic
to a $C^*$-algebra of continuous functions on a compact Hausdorff space.

For a compact Hausdorff space $Z$, we denote by $C(Z)$, 
the $C^*$-algebra of continuous functions on $Z$.

\begin{definition}
\label{def:Higson-compactification}
 The compactification $hX$ of $X$ characterized by the property 
$C(hX) = C_h(X)$ is called the {\itshape Higson compactification}. Its boundary 
$hX\setminus X$ is denoted $\nu X$, and is called the {\itshape Higson corona} 
of $X$.
\end{definition}

\begin{definition}
\label{def:coarse-map}
Let $f\colon X\to Y$ be a map between proper metric spaces. We say that
\begin{itemize}
\item $f$ is {\itshape proper} if for any bounded set $B\subset Y$, the preimage $f^{-1}(B)$ is 
       bounded,
\item $f$ is {\itshape bornologous} if for any $R>0$ there exists $S>0$ such that
       for any $x,x'\in X$, if $d_X(x,x')\leq R$ then $d_Y(f(x),f(x'))\leq S$,
\item $f$ is {\itshape coarse} if $f$ is proper and bornologous.
\end{itemize}

We say that two maps $f_1,f_2\colon X\to Y$ are {\itshape close} if there exists $C\ge 0$ such that 
$d_Y(f_1(x),f_2(x))\le C$ for any $x\in X$.  

The {\itshape coarse category} is a category whose objects are proper metric spaces
and whose morphisms are close classes of coarse maps.  
The spaces $X$ and $Y$ are {\itshape coarsely equivalent} if 
they are isomorphic in the coarse category, that is, 
there exist coarse maps $f\colon X\to Y$ and $g\colon Y\to X$ 
such that $g\circ f$ and $f\circ g$ are close to
the identity maps on $X$ and on $Y$, respectively. 
Such a map $f$ is called a {\itshape coarse equivalence}.
\end{definition}

The assignment $X\mapsto \nu X$ is a functor from the coarse category to
the category of compact Hausdorff spaces. 
For details, see \cite[Section 2.3]{MR2007488} or 
\cite[Section 5.1]{MR1147350}.

\begin{definition}
\label{def:corona}
Let $X$ be a proper metric space.
A {\itshape corona} of $X$ is a pair $(W,\zeta)$ of a compact metrizable space
$W$ and a continuous map $\zeta \colon \nu X\to W$. 
\end{definition}

\begin{definition}
Let $X$ be a proper metric space. Let $\overline{X}$ be a compact
metrizable space containing $X$ as an open set.  Then $\overline{X}$ is
called a {\itshape coarse compactification} of $X$ if the identity
of $X$ extends to a continuous map $hX\to \overline{X}$.
\end{definition}

We remark that a coarse compactification $\overline{X}$ of $X$ is not
necessarily a compactification of $X$.
Indeed we permit the case where $X$ is not dense in $\overline{X}$.
We also remark that any coarse compactification is a
compactification at infinity.

If $\overline{X}$ is a coarse compactification, 
then $\dX:= \overline{X} \setminus X$ is a corona of $X$.
On the other hand, if $(W,\zeta)$ is a corona of $X$, then the space
$X\cup W$, obtained by gluing $W$ to $hX$ by $\zeta$, is a coarse compactification.
The following lemma is an immediate consequence of the definition.
\begin{lemma}
\label{lem:coarse-cptf}
Let $X$ be a proper metric space.
Let $\overline{X}$ be a compact metrizable space containing $X$ as an open set.

Then $\overline{X}$ is a coarse compactification of $X$ if and only if
for any $f\in C(\overline{X})$, the restriction of 
$f$ to $X$ is a Higson function, that is, $f|_X \in C_h(X)$.
\end{lemma}

\begin{remark}
\label{rem:pullback-of-corona} Let $X$ and $Y$ be proper metric spaces
 and let $f\colon X\to Y$ be a coarse map.  If $(W,\zeta)$ is a corona
 of $Y$, then $(W,\zeta \circ \nu f)$ is a corona of $X$. Here $\nu
 f\colon \nu X\to \nu Y$ is a map induced by $f$. Especially we have a
 coarse compactification $X\cup W$.
\end{remark}

\section{Corona of a product space}
\label{sec:corona-product-space}
\subsection{Join}
\label{sec:join-corona}
Let $\{Z_i\}_{i=1}^n$ be a finite sequence of Hausdorff spaces. 
We consider a subset $\Omega$ of $\prod_{i=1}^n([0,1]\times Z_i)$,
\[
 \Omega:=\{(t_1,x_1,\dots,t_n,x_n)\in \prod_{i=1}^n([0,1]\times Z_i): \sum_{i=1}^n t_i= 1\}.
\]
We define an equivalence relation $\sim$ on $\Omega$ as follows.  
For $\vect{x}:=(t_1,x_1,\dots,t_n,x_n)$ and 
$\vect{y}:= (s_1,y_1,\dots,s_n,y_n)$, we have $\vect{x}\sim \vect{y}$ if 
$t_i = s_i$ for all $i$ and $t_i = s_i = 0$ for all $i$ such that $x_i \neq y_i$.
The {\itshape join} of $Z_1,\dots,Z_n$, denoted by $Z_1\star \dots \star Z_n$, 
is the quotient of $\Omega$ by this relation.
We denote the equivalence class of $(t_1,x_1,\dots,t_n,x_n)$ by 
$t_1x_1\oplus \dots \oplus t_nx_n$ or just $\oplus t_ix_i$.

\subsection{Embedding of a product space}
\label{sec:compactification} For $i\in \{1,\dots,n\}$, let $X_i$ be a 
metric space. We equip $\prod_{i=1}^n X_i$ with an $l^1$-metric.
Let $\overline{X_i}$ be a compactification at infinity of $X_i$.
 Set $\dX_i: = \overline{X_i}\setminus X_i$.
We fix a base point $o_i$ of $X_i$.
For $(x_1,\dots,x_n) \in \prod_{i=1}^n X_i$, we set 
\begin{align*}
 \abs{x_i}&:= d(x_i,o_i),\\
 t_i(x_1,\dots,x_n)&:= \frac{1+\abs{x_i}}{\sum_{j=1}^n (1+\abs{x_j})}.
\end{align*}
We define an embedding 
$\iota\colon \prod_{i=1}^n X_i\hookrightarrow 
\overline{X_1}\star \dots \star \overline{X_n}$ as 
 $\iota(x_1,\dots,x_n):=t_1x_1\oplus \dots \oplus t_nx_n$
here $t_i:= t_i(x_1,\dots,x_n)$.
Note that $\dX_1 \star \dots \star \dX_n$ is naturally regarded
as a subspace of $\overline{X_1}\star \dots \star \overline{X_n}$.
\begin{definition}
 We denote by $\overline{\prod_{i=1}^n X_i}$ the closure of $\iota(\prod_{i=1}^n X_i)$ in 
$\overline{X_1}\star \dots \star \overline{X_n}$.
\end{definition}

\begin{lemma}
\label{lem:contained-in}
 The complement 
$\overline{\prod_{i=1}^n X_i}\setminus \iota(\prod_{i=1}^n X_i)$ 
is contained in the compact set $\dX_1 \star \dots \star \dX_n$.  
Therefore both of the compact spaces $\overline{\prod_{i=1}^n X_i}$ and 
$\iota(\prod_{i=1}^n X_i) \cup \dX_1 \star \dots \star \dX_n$ are
 compactifications of $\prod_{i=1}^n X_i$ at infinity.
\end{lemma}

\begin{proof}
We take a sequence $\{\oplus t_i^kx_i^k\}_k$ of $\iota(\prod_{i=1}^n X_i)$
tending to 
$\oplus t_ix_i\in \overline{\prod_{i=1}^n X_i}\setminus
\iota(\prod_{i=1}^n X_i)$.
Then we have $t_i^k\to t_i$ for any $i$ and
$x_i^k\to x_i$ for any $i$ such that  $t_i\neq 0$. Assume that 
$\oplus t_ix_i\notin \dX_1 \star \dots \star \dX_n$.
Then we have $j$ such that $t_j\neq 0$ and $x_j\in X_j$.
Since $|x_j|<\infty$, we have $\sup_k|x_j^k|<\infty$.
This implies $\sup_k|x_i^k|<\infty$ for every $i$.
Then we have $x_i\in X_i$ for every $i$ 
by the definition of compactification at infinity.
This contradicts the fact that $\oplus t_ix_i\not\in \iota(\prod_{i=1}^n X_i)$.
\end{proof}

\begin{definition}
\label{def:visible}
 Let $X$ be a metric space with a base point $o\in X$. Let $\overline{X}$ be a 
 compactification of $X$ at infinity. Set $\dX:= \overline{X}\setminus X$.
 We say that $\dX$ is {\itshape visible from $o$} if for any $x\in \dX$, there exists a continuous map
 $\gamma\colon [0,a]\to \overline{X}=X\cup \dX$ such that 
$\gamma (0) = o, \, \gamma (a) = x$ and $\gamma(t)\in X$ for all $0\leq t< a$.
We say that $\partial X$ is {\itshape visible} if there exist a point $o\in X$
such that $\partial X$ is visible from $o$.
\end{definition}

\begin{remark}\label{(j)}
Note that $|\gamma(t)|$ tends to infinity as $t\to a$ by the
definition of compactification at infinity.
\end{remark}

The Gromov boundary of a proper geodesic hyperbolic space is visible. 
See Section~\ref{sec:busemann-def},~\ref{sec:cone-def} for other examples.



\begin{lemma} 
Let $X_i$ be metric spaces.
 Suppose that $\dX_i$ is visible for any $1\leq i\leq n$. Then
$\overline{\prod_{i=1}^n X_i}
\setminus \prod_{i=1}^n X_i = \dX_1 \star \dots \star \dX_n$.
\end{lemma}
\begin{proof}
We show that the join $\dX_1 \star \dots \star \dX_n$ is contained 
in the closure of $\prod_{i=1}^n X_i$.
We choose a point
\[
 (t_1,x_1,\dots,t_n,x_n)\in \prod_{i=1}^n (0,1)\cap \mathbb{Q} \times \dX_i
\]
with $\sum_{i=1}^{n}t_i = 1$. Since $\dX_i$ is visible from a point 
$o_i \in X_i$ and Remark \ref{(j)},
, we can find a sequence
$\{x_i^{(k)}\}_{k\in \N}$ such that $x_i^{(k)}\to x_i$ as $k\to \infty$ and 
$\abs{x_i^{(k)}} = k$ for any $k\in \N$. 
Choose $p_1,\dots, p_n \in \N$ so that $t_i = p_i/(\sum_{j=1}^n p_j)$. 
It is easy to see that $t_i(x_1^{(p_1k)},\dots, x_n^{(p_nk)}) \to t_i$ 
as $k\to \infty$. Therefore we have 
\[
 \iota(x_1^{(p_1k)},\dots, x_n^{(p_nk)}) \to t_1x_1\oplus \dots \oplus t_nx_n \in 
 \dX_1\star \dots \star \dX_n.
\]
\end{proof}

\subsection{Higson functions}
In this section we prove the following, which implies Proposition~\ref{prop:product-corona-join}. 
\begin{proposition}
\label{prop:join-is-coarse-cptf}
Let $\{X_i\}_{i=1}^n$ be a finite sequence of proper metric spaces. Let $\overline{X_i}$ be a
coarse compactification of $X_i$. Then the compactification $\overline{\prod_{i=1}^n X_i}$
is a coarse compactification. 
Especially, 
$\dX_1 \star \dots \star \dX_n$ is a corona of 
$X_1\times \dots\times X_n$.
\end{proposition}

We prove that $\overline{\prod_{i=1}^n X_i}$ is a coarse
compactification of $\prod_{i=1}^n X_i$.
Then we know that $\partial X_1\star\dots\star\partial X_n$ is a corona
because $\overline{\prod_{i=1}^n X_i}\setminus \prod_{i=1}^n X_i$
is contained in $\partial X_1\star\dots\star\partial X_n$
by Lemma~\ref{lem:contained-in}.

By Lemma~\ref{lem:coarse-cptf}, it is enough to 
show that for all $F\in C(\overline{\prod_{i=1}^n X_i})$, the restriction of $F$ to 
$\prod_{i=1}^n X_i$ is a Higson function.

For $f_i\in C(\overline{X_i})$, we define $F_i(f_i):\overline{\prod_{i=1}^n X_i}\to \C$ by
\begin{align*}
 F_i(f_i)(t_1x_1\oplus \dots \oplus t_nx_n)
 := t_if_i(x_i).
\end{align*}
For $(x_1,\dots,x_n)\in \prod_{i=1}^n X_i$, 
we abbreviate $F_i(f_i)(\iota(x_1,\dots,x_n))$ to 
$F_i(f_i)(x_1,\dots,x_n)$. That is,
\[
 F_i(f_i)(x_1,\dots,x_n):= t_i(x_1,\dots,x_n)f_i(x_i).
\]
Let $\mathbf{A}$ be a $*$-sub algebra of $C(\overline{\prod_{i=1}^n X_i})$ 
generated by $\{F_i(f_i):f_i\in C(\overline{X_i}),\, 1\leq i\leq n\}$.
\begin{lemma}
The algebra  $\mathbf{A}$ separates points in $\overline{\prod_{i=1}^n X_i}$.
\end{lemma}
\begin{proof}
 Let $\oplus t_ix_i\neq \oplus s_iy_i \in \overline{\prod_{i=1}^n X_i}$.
First we assume $t_j \neq s_j$ for some $j\in \{1,\dots,n\}$.
Then $F_j(1)$ separates $\oplus t_ix_i$ and  $\oplus s_iy_i$.
Here $1$ denotes a constant function on $X_j$.

So we assume $(t_1,\dots,t_n) = (s_1,\dots, s_n)$. 
Then there exists $j$ such that $t_j = s_j \neq 0$ and $x_j\neq y_j$. 
There exists $f_j\in C(\overline{X_j})$ such that $f_j(x_j)=0$ and
 $f_j(y_j)= 1$. Thus we have
\begin{align*}
 &F_j(f_j)(\oplus t_ix_i)= t_jf_j(x_j) = 0,\\
 &F_j(f_j)(\oplus s_iy_i)= s_jf_j(x_j) = s_j \neq 0.
\end{align*}
\end{proof}
Therefore $\mathbf{A}$ is dense in $C(\overline{\prod_{i=1}^n X_i})$. 
Then to prove Proposition~\ref{prop:join-is-coarse-cptf}, 
it is enough to show that every $F_i(f_i)$ is a Higson function.

\begin{lemma}
Let $f_i\in C(\overline{X_i})$. Then $F_i(f_i)$ is a Higson function on $\prod_{i=1}^n X_i$.
\end{lemma}

\begin{proof}
 Since $f_i$ is a Higson function on $X_i$, for any $R>0$ and any $\epsilon>0$,
 there exists $K>0$ such that, for any $x,x'\in X_i$ with $\abs{x}>K$ and
 $d(x,x')< R$, we have $\abs{f_i(x)-f_i(x')}< \epsilon$.
Here we suppose $K> \max\{\norm{f_i}R/\epsilon, 2\norm{f_i}/\epsilon +1 \}$.

Let $\vect{x}= (x_1,\dots,x_n), \vect{y}=(y_1\dots,y_n) \in \prod_{i=1}^n X_i$ 
such that $\sum_{i=1}^n \abs{x_i}> K^2$ and 
$d(\vect{x},\vect{y}) = \sum_{i=1}^n d(x_i,y_i) < R$.
Set $L:=\sum_{i=1}^n (1+\abs{x_i})$ and $L':=\sum_{i=1}^n (1+\abs{y_i})$. Then we have
\begin{align*}
  &\abs{L-L'}\leq \sum_{i=1}^n \abs{\abs{x_i}-\abs{y_i}}\leq 
\sum_{i=1}^n d(x_i,y_i)<R,\\
 & \abs{\abs{x_i} - \abs{y_i}}< R,\\
 & \abs{y_i}\leq L'.
\end{align*}
Thus we have
\begin{align*}
 \abs{t_i(\vect{x})-t_i(\vect{y})} &= 
\left|\frac{1}{LL'}(\abs{x_i}L'-\abs{y_i}L+ L' - L)\right|\\
 &= \frac{1}{LL'}\left|(\abs{x_i}-\abs{y_i})L' +
 \abs{y_i}(L'-L) + L'-L\right|\\
 &\leq\frac{3R}{L} < \frac{3\epsilon}{\norm{f_i}}.
\end{align*}
 First we assume $\abs{x_i}>K$. Then we have
\begin{align*}
 \abs{F_i(f_i)(\vect{x})-F_i(f_i)(\vect{y})} &=
 \abs{t_i(\vect{x})f_i(x_i)- t_i(\vect{y})f_i(y_i)}\\
 &\leq \abs{t_i(\vect{x})}\abs{f_i(x_i)-f_i(y_i)} 
 + \abs{f_i(y_i)}\abs{t_i(\vect{x})-t_i(\vect{y})} \\
 &\leq 4\epsilon.
\end{align*}
Now we assume $\abs{x_i}\leq  K$. Then we have 
\begin{align*}
 \abs{t_i(\vect{x})} &= \frac{1+\abs{x_i}}{\sum_{i=1}^n (1+\abs{x_i})} \leq  \frac{1+K}{K^2-K}
 \leq \frac{2}{K-1}< \frac{\epsilon}{\norm{f_i}}.
\end{align*}
Therefore
\begin{align*}
 \abs{F_i(f_i)(\vect{x})-F_i(f_i)(\vect{y})} &=
 \abs{t_i(\vect{x})f_i(x_i)-t_i(\vect{y})f_i(y_i)}\\
 &\leq 2\norm{f_i}\abs{t_i(\vect{x})} + 
\norm{f_i}\abs{t_i(\vect{x})-t_i(\vect{y})} \\
 &\leq 5\epsilon .
\end{align*}
\end{proof}

\begin{example}
 Set $X_1=\R^n$, $X_2=\R^m$, and, $\dX_1 = S^{n-1}$, $\dX_2 = S^{m-1}$.
Then the join $S^{n-1}\star S^{m-1}$ is homeomorphic to $S^{n+m-1}$.
\end{example}

\section{Corona of Busemann spaces and open cones}
\label{sec:corona-busem-spac}
\subsection{Busemann spaces}
We recall the definition of Busemann spaces. For details, see \cite[Chapter 8]{MR2132506}.
\label{sec:busemann-def}
\begin{definition}
 A geodesic space $X$ is said to be a {\itshape Busemann space} if for any two 
geodesics $\gamma\colon [a,b]\to X$ and $\gamma'\colon [a',b']\to X$, 
the map $D_{\gamma,\gamma'}\colon [a,b]\times [a',b']\to \R$ defined by
\[
 D_{\gamma,\gamma'(t,t')}:=\abs{\gamma(t) - \gamma(t') }
\]
is convex.
\end{definition}
We remark that if $X$ is a Busemann space, then for any points $p,q\in X$, 
there exists a unique geodesic connecting $p$ and $q$. 
CAT$(0)$-spaces and strictly convex Banach spaces like $l^p$-spaces ($1<p<\infty$) are 
typical examples of Busemann spaces. 

Now we describe the visual boundary for a proper Busemann space.  
Throughout the rest of this section, 
let $X$ be a proper Busemann space. We
denotes by $C([0,\infty),X)$ the space of continuous functions from
$[0,\infty)$ to $X$ equipped with the topology of uniform convergence on
compact sets.  Let $\partial X\subset C([0,\infty),X)$ be the subspace of
geodesic rays with the base point $o$.  We call $\partial X$ as the
{\itshape visual boundary} of $X$.

We define an embedding $\varphi\colon X \hookrightarrow C([0,\infty),X)$ 
as follows. For $x\in X$, we denote by $\hat{x}$ the geodesic 
from $o$ to $x$ with the parameter $[0,\abs{x}]$.
Then we define $\varphi(x)(t):= \hat{x}(t)$ if $t\leq \abs{x}$ and 
$\varphi(x)(t):= x$ if $t> \abs{x}$. We call $\varphi(x)$ an 
\itshape extended geodesic.

We identify $X$ with $\varphi(X)$.
Especially, for $x\in X$ and $t\in [0,\infty)$, 
we abbreviate $\varphi(x)(t)$ to $x(t)$. We remark that for any 
$x\in \overline{X}\subset C([0,\infty),X)$, it extends to a continuous map 
$x\colon[0,\infty]\to \overline{X}$ by setting $x(\infty)= x$.
We also remark that $\abs{x(t)} = \min\{\abs{x},t\}$.

\begin{proposition}
The space $\overline{X}:= \varphi(X) \cup \partial X \subset C([0,\infty),X)$
is a compactification of $X$, which is visible  from $o$ 
and contractible. Moreover, it is a coarse compactification, so
the visual boundary $\partial X$ is a corona of $X$.
\end{proposition}

\begin{proof}
It is clear that $\overline{X}$ is visible from $o$.
The statements follow from \cite[Proposition 10.2.4]{MR2132506} and
\cite[Section 2]{Busemann_cBC}. See also \cite[Section 4.6]{WillettThesis}.
\end{proof}

\begin{lemma}
\label{lem:unif-conv}
 Let $\{x_n\}$ be a sequence in $\overline{X}$. We suppose that $x_n$ converges to 
$x_\infty \in \overline{X}$. For any sequence $\{t_n\}$ in $[0,\infty]$, 
if $t_n \to t_\infty \in [0,\infty]$, we have $x_n(t_n) \to x_\infty(t_\infty)$.
\end{lemma}
\begin{proof}
 Since $x_n \to x_\infty$ in the topology of uniformly convergence on compact sets, 
 it follows that $x_n(t_n) \to x_\infty(t_\infty)$.
\end{proof}

\subsection{Open cone}
\label{sec:cone-def}
Let $W$ be a compact subset of the unit sphere in a
separable Hilbert space $l^2$.
The {\itshape open cone} on $W$, denoted $\OW$, is the set of all non-negative
multiples of points in $W$. That is, $\OW:= \{tx\in l^2: t\in [0,\infty),\, x\in W\}$. 
We call the origin $o$ of $l^2$ as the {\itshape apex} of $\OW$.

For $d\geq 0$, we denote by $\Pen(\OW,d)$ 
the closed $d$-neighborhood of $\OW$ in $l^2$. We construct a compactification at infinity. 
Let $\lambda \colon [0,\infty)\to [0,1),\, \lambda(t):= t/(1+t)$ be a homeomorphism. 
The continuous map 
\[
 \Lambda\colon \Pen(\OW,d) \to l^2,\, \Lambda(v):= 
 \lambda(\norm{v})\frac{v}{\norm{v}}
\]
is a homeomorphism onto the image. 
Then we identify the union $\Lambda(\Pen(\OW,d))\cup W$ with 
$(\Pen(\OW,d))\cup W$, which we denote by 
$\overline{\Pen(\OW,d)}$. 

Now, as in the case of the Busemann spaces, we identify a point 
$x\in \overline{\Pen(\OW,d)}$ and the
straight line from the apex $o$ to $x$, which we call an 
{\itshape extended geodesic. }
That is, 
for $x\in \Pen(\OW,d)$ and $t\in [0,\infty]$ set 
\[
 x(t):= \min\{t,\norm{x}\} \frac{x}{\norm{x}} \in \Pen(\OW,d).
\]
For $w\in W$ and $t\in [0,\infty]$, if $t<\infty$ then set
\[
 w(t):= tw \in \Pen(\OW,d),
\]
and if $t=\infty$, set  $w(\infty):= w \in W$.
\begin{proposition}
The space $\overline{\Pen(\OW,d)}$
is a compactification of $\Pen(\OW,d)$ at infinity, which is visible  from $o$ 
and contractible. Moreover, if $d=0$, then 
$\overline{\Pen(\OW,0)}=\overline{\OW}$ is a coarse compactification of $\OW$, 
so $W$ is a corona of $\OW$.
\end{proposition}

\begin{proof}
The former part is clear. The latter part can be proved straightforwardly.
See also \cite[Lemma 4.5.3]{WillettThesis}.
\end{proof}

The following, which corresponds to Lemma~\ref{lem:unif-conv}, is trivial.
\begin{lemma}
\label{lem:unif-conv-cone}
 Let $\{x_n\}$ be a sequence in  $\overline{\Pen(\OW,d)}$. 
We suppose that $x_n$ converges to 
$x_\infty \in \overline{\Pen(\OW,d)}$. For any sequence $\{t_n\}$ in $[0,\infty]$, 
if $t_n \to t_\infty \in [0,\infty]$, we have $x_n(t_n) \to x_\infty(t_\infty)$.
\end{lemma}

\section{Contractibility of the compactification at infinity of a product space}
\label{sec:contr-comp-at}
In this section, let $(X_i,o_i,W_i)$ be as follows.
\begin{itemize}
 \item $X_i$ is a Busemann space, $o_i$ is a base point, and $W_i$ is
       the visual boundary, or,
 \item $W_i$ is a compact metrizable space embedded in the unit sphere
       of a Hilbert space, $X_i:= \Pen(\OW,d)$ for some $d\geq 0$ and
       $o_i$ is the apex.
\end{itemize}
For each $i$, we have a visible compactification at infinity 
$\overline{X_i}:= X_i\cup W_i$.
As in Section~\ref{sec:busemann-def} and Section~\ref{sec:cone-def},
we identify a point $x_i$ in $\overline{X_i}$ and the extended geodesic connecting 
$o_i$ and $x_i$.

In this section, we use $[0,\infty]$ as the parameter of homotopies.
Each $\overline{X_i}$ has the following contraction to $o_i$.
\[
 \overline{X_i}\times [0,\infty] \to \overline{X_i} 
\colon (x_i,t) \mapsto x_i(t).
\]

Let $\iota\colon \prod_{i=1}^n X_i\hookrightarrow \overline{X_1}\star \dots \star
\overline{X_n}$ be the embedding defined in Section~\ref{sec:compactification}.
We have a compactification at infinity
$\overline{\prod_{i=1}^n X_i} = \iota(\prod_{i=1}^n X_i)\cup W_1\star \dots \star W_n$ 
by Section~\ref{sec:compactification}.

\begin{theorem}
\label{thm:contractible}
 The compactification at infinity $\overline{\prod_{i=1}^n X_i}$ is contractible.
\end{theorem}
We will construct a map 
$H\colon \overline{\prod_{i=1}^n X_i}\times [0,\infty] \to \overline{\prod_{i=1}^n X_i}$
such that
\begin{enumerate}[$($i$)$]
 \item $H$ is continuous,
\label{item:1}
 \item $H(-,0)$ is a constant map to $\oplus (1/n)o_i$,
\label{item:2}
 \item $H(-,\infty)$ is the identity of $\overline{\prod_{i=1}^n X_i}$.
\label{item:3}
\end{enumerate}

\subsection{Construction of $H$}
We define $H$ as follows. First, for 
$(\oplus t_ix_i,t) \in \iota(\prod_{i=1}^n X_i) \times [0,\infty]$, 
we set 
\begin{align*}
 H(\oplus t_ix_i,t):= \oplus 
\frac{1+\abs{x_i(t_i t)}}{\sum_{j=1}^n (1+\abs{x_j(t_j t)})} x_i(t_i t).
\end{align*}
Since $\abs{x_j(t_j t)}\leq \abs{x_j}<\infty$, we have 
$H(\oplus t_ix_i,t) \in \iota(\prod_{i=1}^n X_i)$.
Next, for $(\oplus t_i w_i,t)\in W_1\star \dots \star W_n
\times [0,\infty)$, we set
\begin{align*}
 H(\oplus t_i w_i,t)&:= \oplus 
\frac{1+\abs{w_i(t_i t)}}{\sum_{j=1}^n (1+\abs{w_j(t_j t)})} w_i(t_i t)\\
 &= \oplus \frac{1+t_i t}{\sum_{j=1}^n (1+t_j t)} w_i(t_i t)\\
 &= \oplus \frac{1+t_i t}{n+t} w_i(t_i t).
\end{align*}
Here if $t_i=0$, then $w_i(t_it)=o_i$, which implies that $H$ is well-defined on 
$W_1\star \dots \star W_n \times [0,\infty)$.

Since $t< \infty$, we have $H(\oplus t_i w_i,t) \in \iota(\prod_{i=1}^n X_i)$.
Finally, for $(\oplus t_i w_i,\infty)
\in W_1\star \dots \star W_n \times \{\infty\}$, we set 
$H(\oplus t_i w_i,\infty):= \oplus t_i w_i$.
Then (\ref{item:2}) and (\ref{item:3}) are trivial. We will prove (\ref{item:1}).
We note that the restrictions of $H$ to the subspaces 
$\iota(\prod_{i=1}^n X_i)\times [0,\infty]$ and 
$W_1\star \dots \star W_n \times [0,\infty)$ are both continuous.
Since $\iota(\prod_{i=1}^n X_i)\times [0,\infty]$ is open in 
$\overline{\prod_{i=1}^n X_i}\times [0,\infty]$, 
it is sufficient to prove that $H$ is continuous at each point of
$W_1\star \dots \star W_n \times [0,\infty]$.

\subsection{Continuity at $W_1\star \dots \star W_n \times [0,\infty)$}
We choose a point 
$(\oplus t_i w_i, t) \in W_1\star \dots \star W_n\times [0,\infty)$. 
We can assume without loss of generality that there exists 
$m\in \{1,\dots,n\}$ such that $t_i\neq 0$ if $i\leq m$ and 
$t_i = 0$ if $i> m$. 
Let $\{(\oplus t_i^{(k)} x_i^{(k)},t^{(k)})\}_k$ 
be a sequence which converges to $(\oplus t_i w_i, t)$. We can assume without 
loss of generality that 
$(\oplus t_i^{(k)} x_i^{(k)},t^{(k)})\in \iota(\prod_{i=1}^n X_i)\times [0,\infty)$,
since $\overline{\prod_{i=1}^n X_i}\times [0,\infty)$
is open in $\overline{\prod_{i=1}^n X_i}\times [0,\infty]$ 
and $H|_{W_1\star \dots \star W_n \times [0,\infty)}$ 
is continuous. 
Here we remark that $t_i^{(k)}= (1+\abs{x_i^{(k)}})/\sum_{j=1}^n(1+\abs{x_j^{(k)}})$.
Since $(\oplus t_i^{(k)} x_i^{(k)},t^{(k)}) \to (\oplus t_i w_i,t)$, we have
\begin{align*}
  & x_i^{(k)} \to w_i \; \text{ if } 1\leq i\leq m,\\
  & t^{(k)} \to t,\\
  & t_i^{(k)} = \frac{1+\abs{x_i^{(k)}}}{\sum_{j=1}^n (1+\abs{x_i^{(k)}})} \to t_i, \; 
 \text{ for all }1\leq i\leq n.
\end{align*}

Since $t_i t< \infty$, we have $t_i^{(k)} t^{(k)} \to t_i t$.
If $1\leq i\leq m$, 
by Lemma~\ref{lem:unif-conv} and Lemma~\ref{lem:unif-conv-cone}, we have
we have 
$x_i^{(k)}(t_i^{(k)} t^{(k)}) \to w_i(t_i t)$, especially,  
$\abs{x_i^{(k)}(t_i^{(k)} t^{(k)})} \to \abs{w_i(t_i t)}= t_i t$.
If $m< i\leq n$, then $t_i^{(k)}t^{(k)} \to 0$, 
thus $x_i^{(k)}(t_i^{(k)} t^{(k)}) \to o_i$, especially, $\abs{x_i^{(k)}(t_i^{(k)} t^{(k)})} \to 0$.
Then we have
\begin{align*}
 & x_i^{(k)}(t_i^{(k)} t^{(k)}) \to \begin{cases}
			       w_i(t_i t) & \text{ if } 1\leq i\leq m,\\
			       o_i & \text{ if } i> m,
			      \end{cases}\\
 & \frac{1+\abs{x_i^{(k)}(t_i^{(k)}t^{(k)})}}
{\sum_{j=1}^n (1+\abs{x_j^{(k)}(t_j^{(k)}t^{(k)})})}
 \to \frac{1+t_i t}{n+t}.
\end{align*}
Therefore $H(\oplus t_i^{(k)} x_i^{(k)},t^{(k)})\to H(\oplus t_i w_i, t)$.

\subsection{Continuity at $W_1\star \dots \star W_n \times \{\infty\}$}
We choose a point 
$(\oplus t_i w_i, \infty) \in W_1\star \dots \star W_n\times \{\infty\}$. 
We can assume without loss of generality 
that there exists $m\in \{1,\dots,n\}$ such that $t_i\neq 0$ if $i\leq m$ and 
$t_i = 0$ if $i> m$. Let $\{(\oplus t_i^{(k)} x_i^{(k)},t^{(k)})\}_k$ 
be a sequence which converges to $(\oplus t_i w_i, \infty)$. We can assume without 
loss of generality that 
$(\oplus t_i^{(k)} x_i^{(k)},t^{(k)})\in \iota(\prod_{i=1}^n X_i)\times [0,\infty]$,
since $\iota(\prod_{i=1}^n X_i)\times [0,\infty]$ is open dense in 
$\overline{\prod_{i=1}^n X_i}\times [0,\infty]$
and $H$ is continuous at each point of 
$W_1\star \dots \star W_n \times [0,\infty)$. 

Here we remark that $t_i^{(k)}= (1+\abs{x_i^{(k)}})/\sum_{j=1}^n(1+\abs{x_j^{(k)}})$.
Since $(\oplus t_i^{(k)} x_i^{(k)},t^{(k)}) \to (\oplus t_i w_i,\infty)$, we have
\begin{align}
\label{eq:6}
  & x_i^{(k)} \to w_i \; \text{ if } 1\leq i\leq m,\\
\label{eq:5}
  & t^{(k)} \to \infty,\\
\label{eq:2}
  & t_i^{(k)} = \frac{1+\abs{x_i^{(k)}}}{\sum_{j=1}^n (1+\abs{x_i^{(k)}})} \to t_i, \; 
\text{ for all }1\leq i\leq n.
\end{align}
From (\ref{eq:2}), for $1\leq p\leq m$ and $1\leq q\leq n$, we have 
\begin{align}
\label{eq:3}
 \frac{1+\abs{x_q^{(k)}}}{1+\abs{x_p^{(k)}}} 
= \frac{t_q^{(k)}}{t_p^{(k)}} \to \frac{t_q}{t_p}.
\end{align}
To show the continuity of $H$ at $(\oplus t_i w_i,\infty)$, it is enough to show the following.
\begin{lemma}
 Under the above setting we have
\begin{align}
\label{eq:4}
 & x_i^{(k)}(t_i^{(k)} t^{(k)}) \to w_i\;  \text{ if } 1\leq i\leq m\\
\label{eq:99}
 & \frac{1+\abs{x_i^{(k)}(t_i^{(k)}t^{(k)})}}
{\sum_{j=1}^n (1+\abs{x_j^{(k)}(t_j^{(k)}t^{(k)})})}
 \to t_i, \,  \text{ for all } 1 \leq i \leq n.
\end{align}
\end{lemma}
\begin{proof}
If $i\leq m$, from (\ref{eq:5}), we have $t_i^{(k)}t^{(k)}\to \infty$.
Then (\ref{eq:4}) follows from (\ref{eq:6}), Lemma~\ref{lem:unif-conv}
and Lemma~\ref{lem:unif-conv-cone}. So we will show (\ref{eq:99}).
Let $1 \leq p\leq m$ and $1\leq q\leq n$. Set
\[
 K^{(k)}_{p,q}:= \frac{1+\abs{x_q^{(k)}(t_q^{(k)} t^{(k)})}}{1+\abs{x_p^{(k)}(t_p^{(k)} t^{(k)})}}
 =\frac{1+\min\{\abs{x_q^{(k)}}, t_q^{(k)} t^{(k)}\}}{1+\min\{\abs{x_p^{(k)}}, t_p^{(k)} t^{(k)}\}}.
\]

Then it is easy to see that 
\[
\min\left\{\frac{1+t_q^{(k)}t^{(k)}}{1+t_p^{(k)}t^{(k)}}, 
\frac{1+\abs{x_q^{(k)}}}{1+ \abs{x_p^{(k)}}}\right\} \leq  K^{(k)}_{p,q} 
 \leq   
\max\left\{\frac{1+t_q^{(k)}t^{(k)}}{1+t_p^{(k)}t^{(k)}}, 
\frac{1+\abs{x_q^{(k)}}}{1+ \abs{x_p^{(k)}}}\right\}
\]

Thus by (\ref{eq:5})(\ref{eq:2})(\ref{eq:3}), we have $K^{(k)}_{p,q} \to t_q / t_p$.
Therefore if $1\leq i\leq m$, 
\begin{align*}
 \sum_{j=1}^n \frac{1+ \abs{x_j^{(k)}(t_j^{(k)} t^{(k)})}}{1+ \abs{x_i^{(k)}(t_i^{(k)} t^{(k)})}}
 = \sum_{j=1}^n K_{i,j}^{(k)} \to \frac{1}{t_i}.
\end{align*}
It follows that 
\[
 \frac{1+ \abs{x_i^{(k)}(t_i^{(k)} t^{(k)})}}
{\sum_{j=1}^n (1+ \abs{x_j^{(k)}(t_j^{(k)} t^{(k)})})}
\to t_i.
\]

If $i> m$, we have
\begin{align*}
  \frac{1+ \abs{x_i^{(k)}(t_i^{(k)} t^{(k)})}}
{\sum_{j=1}^n (1+ \abs{x_j^{(k)}(t_j^{(k)} t^{(k)})})}
  \leq \frac{1+ \abs{x_i^{(k)}(t_i^{(k)} t^{(k)})}}{1+ \abs{x_1^{(k)}(t_1^{(k)} t^{(k)})}} 
\to \frac{t_i}{t_1} = 0.
\end{align*}
\end{proof}

\section{Review of coarse algebraic topology}
\label{sec:revi-coarse-algebr}
For a locally compact second countable Hausdorff space $X$, we denotes by $K_*(X)$ the
reduced $K$-homology of the one point compactification of $X$ and call
it the $K$-homology of $X$. That is, $K_*(X):=\rK_*(X\cup \{+\})$.  Here
$+$ denotes a point which is not contained in $X$ and $X\cup \{+\}$ is
the one point compactification of $X$. Then $K_*(-)$ is a generalized
homology theory on the category of locally compact second countable Hausdorff 
spaces in the sense of \cite[Definition 7.1.1]{MR1817560}.
Let $Z$ be a compact metrizable space and $i\colon X\hookrightarrow Z$ 
be an embedding such that $i(X)$ is open in $Z$. We identify $X$ and
$i(X)$. Set $W:= Z\setminus X$.  
Then the boundary homomorphism $K_*(X)\to \rK_{*-1}(W)$ is the connection homomorphism in 
the exact sequence of \cite[Definition 7.1.1(b)]{MR1817560} for $W\subset Z$. 

\begin{definition}[{\cite[(3.13) Definition]{MR1147350}}]
\label{def:anti-Cech-sys}
 Let $X$ be a metric space. Let $\U_1,\U_2,\dots$ be a sequence of
 locally finite covers of $X$. We say that they form {\itshape 
 an anti-\v{C}ech} system if there exists a sequence of real numbers 
$R_n\to \infty$ such that for all $n$, 
\begin{enumerate}
 \item each set $U\in \U_n$ has diameter less than or 
equal to $R_n$, and
 \item the cover $\U_{n+1}$ has a Lebesgue number $\delta_{n+1}$
       greater than or equal to $R_n$, that is, any set of diameter less
       than or equal to $\delta_{n+1}$ is contained in some element of
       $\U_{n+1}$.
\end{enumerate}
\end{definition}
These conditions imply that for each $n$, 
there exists a map
$\varphi_{n}\colon \U_n\to \U_{n+1}$ such that 
$U\subset \varphi_n(U)$ for all $U\in \U_n$. We call $\varphi_{n}$ a
{\itshape coarsening map}. We remark that this map is called a refining map in the
context of \v{C}ech cohomology theory. We denote by $\abs{\U_n}$ the nerve complex of  $\U_n$.
A coarsening map $\varphi_{n}$ induces a proper
simplicial map $\abs{\U_n}\rightarrow \abs{\U_{n+1}}$,
which we also denote by the same symbol $\varphi_{n}$ and
also call it a {\itshape coarsening map}.

By using a partition of unity, we have a proper continuous map 
$\varphi_0:X\to \abs{\U_1}$, which is also called a {\itshape coarsening map}. 
Then we have the sequence
\[
X\to \abs{U_1}\to \abs{U_2}\to \cdots, 
\]
which we call a {\itshape coarsening sequence} for $X$.
Note that for a proper metric space $X$, 
a coarsening sequence is unique up to strong proper homotopy in the following sense. 

\begin{definition}
Let 
$\{\alpha_n:A_n\to A_{n+1}\}_{n\in \N\cup\{0\}}$ and 
$\{\beta_n:B_n\to B_{n+1}\}_{n\in \N\cup\{0\}}$
be two sequences of locally compact second countable Hausdorff spaces and 
proper continuous maps. 
They are said to be {\itshape properly homotopic} if 
there exist two subsequence $\{k(n)\}_{n\in \N\cup\{0\}}, \{l(n)\}_{n\in \N\cup\{0\}}$ of $\N\cup\{0\}$ and 
two sequences of proper continuous maps $\{\phi_n:A_n\to B_{k(n)}\}_{n\in\N\cup\{0\}}$, 
$\{\psi_n:B_n\to A_{l(n)}\}_{n\in\N\cup\{0\}}$ such that 
$\psi_{k(n)}\circ \phi_n$ is properly homotopic to $\alpha_{l(k(n))}\circ\cdots\circ\alpha_n$
and 
$\phi_{l(n)}\circ \psi_n$ is properly homotopic to $\beta_{k(l(n))}\circ\cdots\circ\beta_n$
for every $n\in \N\cup\{0\}$.

When we suppose $A_0=B_0$, two sequences are said to be {\itshape strongly properly homotopic}
if we can take $k(0), l(0)$ as $0$ and $\phi_0, \psi_0$ as the identity. 
\end{definition}

Let $X$ be a proper metric space. 
Let $X\to \abs{U_1}\to \abs{U_2}\to \cdots$ be 
a coarsening sequence of $X$. 
Then the {\itshape coarse $K$-homology} of $X$, denoted by $KX_*(X)$, 
is an inductive limit
\[
 KX_*(X) := \varinjlim K_*(\abs{\U_i}).
\]
We also have
a natural map $c\colon K_*(X) \to KX_*(X)$, called a
{\itshape coarsening map}. This coarsening map $c$ relates 
the coarse geometry of $X$ with the topology of $X$.
It is known that $KX_*(X)$ is invariant under the coarse equivalence. 
For details, see~\cite[Section 2]{MR1388312}.
We can equip $\abs{\U_i}$ with a proper metric so that $\abs{\U_i}$ is coarsely equivalent to $X$. 
Then we have an {\itshape assembly map}~\cite[(6.1)]{MR1388312}
\[
 A\colon K_*(\abs{\U_i}) \to K_*(C^*(\abs{\U_i})).
\]
We remark that $K_*(C^*(\abs{\U_i}))$ is canonically isomorphic to
$K_*(C^*(X))$ for all $i$.  Then by taking the inductive limit, we have
a {\itshape coarse assembly map}
\[
 \mu_X \colon KX_*(X) \to K_*(C^*(X)).
\]

Let $W$ be a corona of $X$.  Since $\abs{\U_i}$ is coarsely equivalent
to $X$, the space $W$ is also a corona of $\abs{\U_i}$.  Then we have the
boundary homomorphism $K_*(\abs{\U_i})\to \rK_{*-1}(W)$.  By taking the
inductive limit, we have the {\itshape transgression map}
\[
 T_W \colon KX_*(X) \to \rK_{*-1}(W).
\]

We prepare some notion and facts  on coarse algebraic topology.
For a metric space $X$, a subspace $A$, and a positive number $R$, we
denote by $\Pen(A;R)$ the closed $R$-neighborhood of $A$ in $X$, that is,
$\Pen(A;R) = \{p\in M: d(p,A)\leq R\}$.
\begin{definition}
 Let $X$ be a proper metric space, and let $A$ and $B$ be closed
 subspaces with $X = A\cup B$. We say that 
 $X= A\cup B$ is a {\itshape coarse excisive
 decomposition}, if for each $R> 0$ there exists some $S> 0$ such that 
\[
 \Pen(A;R)\cap \Pen(B;R) \subset \Pen(A\cap B; S).
\]
\end{definition}

We summarize results in \cite{MR1219916} on
 coarse assembly maps and Mayer-Vietoris sequences as follows:
\begin{theorem}
\label{th:MV}
 Suppose that $X=A \cup B$ is a coarse excisive decomposition. 
Then the following diagram is commutative and horizontal sequences are exact:
\begin{align*}
\xymatrix@1{
\ar[r]& KX_p(A\cap B) \ar[r] \ar[d]
 & KX_p(A) \oplus KX_p(B) \ar[r] \ar[d]
 & KX_p(X) \ar[r] \ar[d] & KX_{p-1}(A\cap B) \ar[r] \ar[d] &\\
\ar[r] & K_p(C^*(A\cap B)) \ar[r]
 & K_p(C^*(A))\oplus K_p(C^*(B)) \ar[r]
 & K_p(C^*(X)) \ar[r] & K_{p-1}(C^*(A\cap B)) \ar[r]. &
}
\end{align*}
Here vertical arrows are coarse assembly maps.
\end{theorem}

There are several definitions of coarse homotopy. 
We use the following definition from \cite[Section 3]{boundary}.
\begin{definition}
 Let $f,g\colon X\to Y$ be coarse maps between proper metric
 spaces. We say that they are
 {\itshape coarsely homotopic} if there exists a metric subspace 
\[
 Z = \{(x,t):1\leq t\leq T_x\}
\]
of $X\times [1,\infty)$ and a coarse map
 $h\colon Z\to Y$, such that
\begin{enumerate}
 \item the map $x\mapsto T_x$ is bornologous,
 \item $h(x,1) = f(x)$, and
 \item $h(x,T_x) = g(x)$.
\end{enumerate}
Here we equip $X\times [1,\infty)$ with the $l_1$-metric, that is, 
$d_{X\times [1,\infty)}((x,n),(y,m)):= d_X(x,y)+ \abs{n-m}$ for 
$(x,n),(y,m)\in X\times [1,\infty)$, where $d_X$ is the metric on $X$.

Furthermore, if all maps are continuous, then we say that $f$ and $g$ are 
{\itshape continuously coarsely homotopic}.
\end{definition}
Coarse homotopy and continuous coarse homotopy are respectively 
equivalence relations on coarse maps and on continuous coarse maps. 

\begin{lemma}
\label{lem:product-of-chmtpy}
Let $X_1,X_2,Y_1,Y_2$ be proper metric spaces. 
Suppose that $X_i$ and $Y_i$ are coarsely homotopic for $i=1,2$.
Then $X_1\times X_2$ and $Y_1\times Y_2$ are coarsely homotopic.
\end{lemma}
\begin{proof}
 It is easy to see that 
 $X_1\times X_2$ and $X_1\times Y_2$ are coarsely homotopic.
 Also  $X_1\times Y_2$ and $Y_1\times Y_2$ are coarsely homotopic.
 Since coarse homotopy is an equivalence relation,
 $X_1\times X_2$ and $Y_1\times Y_2$ are coarsely homotopic.
\end{proof}

\begin{definition}
\label{def:flasque}
Let $X$ be a space with a proper metric $d$. 
We say that $X$ is {\itshape coarsely flasque}, 
if there exists a coarse map $\phi:X\to X$ such that 
\begin{enumerate}
\item $\phi$ is close to the identity;
\item for any bounded subset $K\subset X$, there exists $N_K\in \N$
such that for any $n\ge N_K$, $\phi^n(X)\cap K=\emptyset$;   
 \item for all $R>0$, there exists $S>0$ such that for all $n\in \N$ and
       all $x,y \in X$ with $d(x,y)<R$, we have 
       $(d(\phi^n(x),\phi^n(y)))<S$.
\end{enumerate}
\end{definition}

\begin{lemma}
\label{lem:flasque-vanish}
 Let $X$ be a proper metric space. If $X$ is coarsely flasque, then 
$KX_*(X) = K_*(C^*(X)) = 0$.
\end{lemma}
\begin{proof}
 See \cite[Lemma 3.4]{boundary}.
\end{proof}

\begin{corollary}
\label{prop:chinvariance-of-chomology}
 The coarse $K$-homology, the $K$-theory of the Roe algebra 
 and the coarse assembly map are coarse homotopy invariant.
\end{corollary}

\begin{proof}
 The coarse homotopy invariance is a consequence of Theorem~\ref{th:MV}
 and Lemma~\ref{lem:flasque-vanish}. See arguments in 
\cite[Proof of Proposition 12.4.12]{MR1817560},
\cite[Section 3]{boundary} for detail.
\end{proof}

\begin{definition}[{\cite[Definition 12.4.7]{MR1817560}}]
 A proper metric space $X$ is {\itshape scalable} if there is a
 continuous coarse map $f\colon X\to X$, which is continuously coarsely
 homotopic to the identity map, such that
\[
 d(f(x),f(x'))\leq \frac{1}{2}d(x,x')
\]
for all $x,x'\in X$.
\end{definition}
It is easy to see that the product of scalable spaces is scalable.

\begin{proposition}[{\cite[Theorem 12.4.11]{MR1817560}}]
\label{prop:scalable-cBC}
 Let $X$ be a scalable proper metric space. Then the following
 assembly map is an isomorphism.
\[
 A\colon K_*(X) \to K_*(C^*(X)).
\]
\end{proposition}

\subsection{Comparison between coarse $K$-homology and $K$-homology}
Here we give a condition such that a coarsening map 
$c\colon K_*(X) \to KX_*(X)$ become an isomorphism.

\begin{definition}
Let $\{\alpha_n\colon A_n\to A_{n+1}\}_{n\in \N\cup \{0\}}$ be a
sequence of locally compact second countable Hausdorff spaces and proper
continuous maps.  We call a sequence of proper continuous maps
$\{g_n\colon A_n\to A_0\}_{n\in\N}$ a {\itshape splitting up to proper
homotopy} if $g_n \circ (\alpha_{n-1} \circ\cdots\circ \alpha_0)$ is
properly homotopic to the identity on $A_0$ and $(\alpha_n
\circ\cdots\circ \alpha_0)\circ g_n$ is properly homotopic to
$\alpha_n$.
\end{definition}

By an argument in \cite[Proposition 3.8]{MR1388312}, we have the following.
\begin{lemma}
For the above sequence, we have 
$K_*(A_0)\cong \Im(K_*(A_n)\to K_*(A_{n+1}))\cong  \varinjlim K_*(A_n)$ 
\end{lemma}

\begin{definition}
Let $X$ be a proper metric space. 
Let 
\[
X\to X_1\to X_2\to\cdots
\]
be a sequence of locally compact second countable Hausdorff spaces and proper continuous maps.
When it is strongly properly homotopic to a coarsening sequence for $X$, 
we call it a {\itshape coarsening sequence in a wider sense}.
\end{definition}

For a proper metric space $X$, a coarsening sequence in a wider sense 
is unique up to strong proper homotopy. 
Also if a coarsening sequence for $X$ in a wider sense has a splitting up to proper homotopy, 
then so does every coarsening sequence for $X$ in a wider sense.  

The following is clear.
\begin{lemma}
Let $X$ be a proper metric space. 
When $X\to X_1\to X_2\to\cdots$ is a coarsening sequence for $X$ in a wider sense, 
then $KX_*(X)\cong \varinjlim K_*(X_n)$. 
Moreover if it has a splitting up to proper homotopy, then we have 
$K_*(X)\cong \Im(K_*(X_n)\to K_*(X_{n+1}))\cong KX_*(X)$.
\end{lemma}

\begin{proposition}[Higson-Roe]
\label{lem:EG-admits-split}
 Let $G$ be a finitely generated group. We assume that there exists a
 universal space $\EG$ for proper actions which is a finite
 $G$-simplicial complex. Then a coarsening sequence 
for $\EG$ has a splitting up to proper homotopy.
\end{proposition}
\begin{proof}
 See \cite[Proposition 3.8]{MR1388312}.
\end{proof}

\begin{lemma}
\label{lem:example-split-anti-Cech-sys}
Let $X$ be a proper Busemann space, or
simply connected solvable Lie group with a lattice.
Then a coarsening sequence for $X$ has a splitting up to proper homotopy.
\end{lemma}

\begin{proof}
 See \cite[Section 3]{Busemann_cBC} for the case of Busemann spaces.
If $X$ is simply connected solvable Lie group with a lattice, 
then $X$ is uniformly contractible and of bounded geometry, thus
by \cite[Proposition 3.8]{MR1388312}, we have the conclusion.
\end{proof}

\section{Simply connected solvable Lie groups with lattices}
\label{sec:simply-conn-solv}
In this section, we discuss the coarse geometry of simply connected
solvable Lie groups with lattices. We refer to \cite[Section 4]{GL-spin-fundametalgr} and \cite[Section 6.3]{MR1399087}.
First, we remark a simple lemma in coarse algebraic topology.
\begin{lemma}
\label{lem:a-cup-b}
Let $X$ and $Y$ be proper metric spaces. 
Let $X=A\cup B$ be a coarse excisive decomposition. 
Suppose that $A,B$ are coarsely flasque.
Then we have that 
$\mu_*(Y\times X)$ is an isomorphism if and only if 
so is $\mu_*(Y\times (A\cap B))$. 
\end{lemma}
\begin{proof}
We can easily confirm that $Y\times X=(Y\times A)\cup (Y\times B)$ 
is a coarse excisive decomposition 
and $Y\times A,Y\times B$ are coarsely flasque.
Then the Mayer-Vietoris sequence gives the claim.
\end{proof}

\begin{proposition}
\label{prop:assembly-map-solv}
Let $n$ be a positive integer and 
$\LieG$ be an $n$-dimensional simply connected solvable Lie group 
with a lattice.
We equip $\LieG$ with the proper left invariant metric.
Let $Y$ be a proper metric space. Then $\mu_Y$ is an isomorphism 
if and only if so is $\mu_{Y\times \LieG}$.
\end{proposition}
\begin{proof}
Let $G$ be a lattice of $\LieG$.
Then $M:=\LieG/G$ is called a solvmanifold, 
which is compact. It is known that any solvmanifold has a structure of 
a fiber bundle on $S^1$ with fiber a solvmanifold $L$.
In particular we have a diffeomorphism $f:L\to L$
such that $M$ and $\R\times_\Z L$ are diffeomorphic, 
where $\R\times_{\Z} L$ is the quotient of $\R\times L$ by 
the action $\rho_k:\R\times L\ni (r,l)\to (r+k,f^k(l))\in \R\times L$ for $k\in \Z$. 
We identify $M$ with $\R\times_\Z L$. 
We denote by $\tilde{L}$ the universal cover of $L$.
We can lift $f$ to a diffeomorphism
$\tilde{f}\colon \tilde{L} \to \tilde{L}$. 
Thus we have an action $\tilde{\rho}$ of $\Z$ on 
$\R\times \tilde{L}$ by 
\[
 \tilde{\rho}_k:\R\times \tilde{L}\ni (r,l)\to (r+k,\tilde{f}^k(l))\in \R\times \tilde{L}
 \text{ for } k\in \Z, 
\]
which is a lift of $\rho$.

We construct a $\rho$-equivariant riemannian metric $g$ on $\R\times L$
which satisfies the following.
\begin{itemize}
 \item Subspaces $\R\times \{l\}$ and $\{r\}\times L$ 
are orthogonal at $(r,l)\in \R\times L$,
 \item The restriction of $g$ on $\R\times \{l\}$ is $dr^2$ for every $l\in L$. 
\end{itemize} 
First, we take a riemannian metric $h$ on $L$.
Then the orthogonal sum $dr^2 \oplus h$ 
defines a riemannian metric on $[0,1)\times L$
and also the orthogonal sum $dr^2 \oplus (f^{-1})^*h$ 
defines a riemannian metric on $[0,1)\times L$.
Take smooth functions $\sigma_1, \sigma_2$ on $[0,1)$ such that
$\sigma_1+\sigma_2=1$,
the support of $\sigma_1$ is in $[0,2/3]$ and
the support of $\sigma_2$ is in $[1/3,1)$.
Then the orthogonal sum $dr^2 \oplus (\sigma_1h+\sigma_2(f^{-1})^*h)$
gives a riemannian metric on $[0,1)\times L$. 
Now we define 
\[
 g_{(r,l)}:=dr^2\oplus (\sigma_1(r-[r])(f^{-[r]})^*h
 +\sigma_2(r-[r])(f^{-[r]-1})^*h).
\]
where $[r]$ 
denote the greatest integer which  is smaller than or equal to $r$.
Then $g$ satisfies the above two conditions.

We denote by $h_{r}$ the restriction $g$ on $\{r\}\times L$.
We remark that $g$ gives a riemannian metric $g_M$ on $M$ 
and a $G$-equivariant riemannian metric $\tilde{g}$ on $\LieG$. 
Hence $\tilde{g}$ is a
coarsely equivalent to any proper left invariant metric on $\LieG$.
We also remark that the restriction of $\tilde{g}$ on $\{r\}\times \tilde{L}$, 
which we denote by $\tilde{h}_{r}$, is a lift of $h_r$.
Then $\R\times \tilde{L}=\R_{\le 0}
\times \tilde{L}\cup \R_{\ge 0}\times \tilde{L}$ is a 
coarse excisive decomposition. 
We can also prove that $\R_{\le 0}\times \tilde{L}$ 
and $\R_{\ge 0}\times \tilde{L}$ are coarsely flasque. 
Indeed 
$\tilde{\rho}_{-1}|_{\le 0}: \R_{\le 0}\times \tilde{L}\ni (r,l)\mapsto 
(r-1, f^{-1}(l))\in \R_{\le 0}\times \tilde{L}$
and 
$\tilde{\rho}_1|_{\ge 0}: \R_{\ge 0}\times \tilde{L}\ni (r,l)\mapsto 
(r+1, f(l))\in \R_{\ge 0}\times \tilde{L}$
are isometries, and thus their iterations are also isometries. 
The iterations place any bounded subsets. 
They are also close to the identities, respectively

Now we apply inductively Lemma~\ref{lem:a-cup-b} to the case. 
Then for ever proper metric space $Y$, we have that 
$\mu_*(Y\times \LieG)$ is an isomorphism if and only if so is $\mu_*(Y)$.
\end{proof}

\begin{lemma}
\label{lem:homeo-coarse-map}
Let $\LieG$ be an $n$-dimensional simply connected solvable Lie group 
with a lattice.
Then there exists a homeomorphic coarse map $\phi:\LieG\to \R^n$.
By pulling back the visual boundary $S^{n-1}$, 
we regard it as a corona of $\LieG$. 
Especially, the coarse compactification $\LieG \cup S^{n-1}$ is 
homeomorphic to the unit ball of $\R^n$.
\end{lemma}
\begin{proof}
We inductively construct a desired map.
We consider the setting in the previous proof.
We note that 
$\mathrm{id}_r:(\{r\}\times \tilde{L}, \tilde{h}_{r})\ni(r,l)
\mapsto (0,l)\in(\{0\}\times \tilde{L}, \tilde{h}_{0})$
is Lipschitz, since
$\tilde{h}_{r}$ is a lift of $h_r$ and $L$ is compact. 
We take a Lipschitz constant $C(r)$ for each $r\in \R$ continuously.

Suppose that we have 
a homeomorphic coarse map $\eta :\tilde{L}\to \R^{n-1}$.
Here we endowed $L$ with a riemannian metric $h$ which is isometric to 
$(\{0\}\times L, h_{0})$. 
Then $\R\times \tilde{L}\ni(r,l)\mapsto (r,\eta(l)/C(r))\in\R^n $
is a homeomorphic coarse map.
\end{proof}

\begin{corollary}
\label{cor:LieG-homeo}
Let $\LieG$ be an $n$-dimensional simply connected solvable Lie group
with a lattice.
Let $\phi:\LieG\to \R^n$ be the homeomorphic coarse map constructed in
the proof of Lemma \ref{lem:homeo-coarse-map}.
Let $Y$ be a proper metric space.
Then $\phi_*:KX_*(Y\times\LieG)\to KX_*(Y\times\R^n)$
is an isomorphism.
\end{corollary}
\begin{proof}
Note that $\phi$ is identified with the homeomorphic coarse map
$\R\times \tilde{L}\ni(r,l)\mapsto (r,\eta(l)/C(r))\in\R^n $.
Then this map clearly preserves the coarse excisive decompositions
$\R\times \tilde{L}=\R_{\le 0}\times \tilde{L}\cup \R_{\ge 0}\times \tilde{L}$
and $\R^n=\R_{\le 0}\times \R^{n-1}\cup \R_{\ge 0}\times \R^{n-1}$.
Since $\R_{\le 0}\times \tilde{L},\R_{\ge 0}\times \tilde{L},\R_{\le
0}\times \R^{n-1},\R_{\ge 0}\times \R^{n-1}$
are coarsely flasque, the claim is inductively proved in view of
Proofs of Lemma 7.1 and Proposition 7.2.
\end{proof}

\section{Proof of Theorem~\ref{th:corona-assembly}}
\label{sec:proof-theor-refth:c}
The following is a key to the proof of Theorem~\ref{th:corona-assembly}.
We refer to \cite[Proof of (4.3) Proposition]{MR1388312}.
\begin{proposition}
\label{prop:coarsening} Let $X$ be a product of finitely many
proper Busemann spaces.  Let Y product of finitely many open cones
over compact metrizable spaces. Then the coarsening map
\[
 c\colon K_*(X  \times Y) \to KX_*(X\times Y)
\]
is an isomorphism.
\end{proposition}

\begin{proof}
For simplicity, we suppose that $X$ is a Busemann space and $Y$ is a
open cone $\OW$ over a compact metrizable space $W$. General case can be
shown by the same way

By Lemma~\ref{lem:example-split-anti-Cech-sys}, we have a coarsening sequence
$X\to \abs{\U_1} \to \abs{\U_2} \to \cdots$.
Indeed \cite[Section 3]{Busemann_cBC} gives for every $i$
a continuous coarse map $g_i:\abs{\U_i}\to X$
which is a coarse equivalence and satisfies 
that the coarsening map $\abs{\U_i}\to \abs{\U_{i+1}}$
is properly homotopic and close to the composite
\begin{align}
\label{eq:X-composite}
\abs{\U_i}\to X\to \abs{\U_{i+1}}.
\end{align}
and that the identity on $X$ is close to the composite
\begin{align}
\label{eq:X-composite'}
X\to \abs{\U_i}\to X.
\end{align}

Take a coarsening sequence $\OW \to \abs{\V_1}\to \abs{V_2}\to \cdots$
as in the proof of
\cite[Proposition 4.3]{MR1388312} or in \cite[Appendix B]{relhypgrp}.
In \cite[Proposition 4.3]{MR1388312}, for every $i$,
a positive integer $d_i$ and
continuous coarse  maps
\[g'_i:\abs{\V_i}\to \Pen(\OW, d_i), \ h'_i:\Pen(\OW, d_i)\to \abs{\V_{i+1}}\]
are given.
By the construction, they are coarse equivalence and satisfy that
the coarsening map $\abs{\V_i}\to \abs{\V_{i+1}}$
is properly homotopic and close to the composite
\begin{align}
\label{eq:cone-composite}
\abs{\V_i}\to \Pen(\OW, d_i)\to \abs{\V_{i+1}},
\end{align}
and the injection $\iota: \OW\to \Pen(\OW, d_i)$
is close to the composite
\begin{align}
\label{eq:cone-composite'}
\OW\to \abs{\V_i}\to \Pen(\OW, d_i).
\end{align}
Hence the coarsening map
$\abs{\U_i} \times \abs{\V_i}\to
\abs{\U_{i+1}}\times \abs{\V_{i+1}}$
is properly homotopic and close to the composite
\begin{align}
\label{eq:composite}
\abs{\U_i}\times \abs{\V_i}\to X \times \Pen(\OW, d_i)
\to \abs{\U_{i+1}} \times \abs{\V_{i+1}}
\end{align}
and $id_X\times\iota:X\times \OW\to X\times \Pen(\OW, d_i)$
is close to the composite
\begin{align}
\label{eq:composite'}
X \times \OW\to \abs{\U_i}\times \abs{\V_i}\to X \times \Pen(\OW, d_i).
\end{align}
Since the continuous coarse map $X \times \OW\to
\abs{\U_i}\times \abs{\V_i}$
extends to the identity on $\partial X \star W$ by 
Remark~\ref{rem:pullback-of-corona}, 
and (\ref{eq:composite'}) is close to
the injection $id_X\times\iota:X\times \OW\to X\times \Pen(\OW, d_i)$,
we can easily confirm that 
$\abs{\U_i}\times \abs{\V_i}\to X \times \Pen(\OW, d_i)$
extends to the identity on $\partial X \star W$.
Moreover since (\ref{eq:composite}) 
is close to the continuous coarse map
$\abs{\U_i} \times \abs{\V_i}\to \abs{\U_{i+1}}\times \abs{\V_{i+1}}$,
we can easily confirm that $X \times \Pen(\OW, d_i)\to \abs{\U_{i+1}}
\times \abs{\V_{i+1}}$
extends to the identity on $\partial X \star W$.
Now we have
\begin{align}
\label{eq:extension-of-coarsening}
\abs{\U_i}\times \abs{\V_i} \cup \partial X\star W\to X\times
\Pen(\OW, d_i) \cup \partial X\star W
\to \abs{\U_{i+1}} \times \abs{\V_{i+1}} \cup \partial X\star W,
\end{align}
\begin{align}
\label{eq:extension-of-coarsening'}
X\times \OW \cup \partial X\star W\to
\abs{\U_i}\times \abs{\V_i} \cup \partial X\star W.
\end{align}
By Theorem~\ref{thm:contractible},
the space 
$\overline{X\times \Pen(\OW, d_i)} = X\times \Pen(\OW,d_i)\cup \dX \star W$
is contractible,
so the map (\ref{eq:extension-of-coarsening}) is null-homotopic. Thus
it induces the null-map of the reduced $K$-homology.
This fact, (\ref{eq:extension-of-coarsening'}) and long exact sequences 
of $K$-homology imply
\begin{align}
\label{diagram:transgression}
\xymatrix{
K_{*}(X \times \OW) \ar[d] \ar[dr]^\cong\\
KX_*(X \times \OW) \ar[r]^\cong &\rK_{*-1}(\partial X \star W).
}
\end{align}
Then the proposition follows from the commutative diagram.
\end{proof}

\begin{proof}[Proof of Theorem~\ref{th:corona-assembly}]
Let $\{(X_i,o_i,W_i)\}_{i=1}^n$ be a finite sequence of proper metric spaces, 
base points, and compact metrizable spaces.
We suppose that there exist integers $0\leq k\leq l \leq m\leq n$ such that 
\begin{itemize}
 \item For $1\leq i\leq k$, 
       $X_i$ is a geodesic Gromov hyperbolic space, $o_i\in X_i$ and $W_i$ is the Gromov boundary,
 \item For $k< i\leq l$, 
       $X_i = \OW_i$ 
       is an open cone over a compact metrizable space 
       $W_i$ with the apex $o_i$,
 \item For $l< i\leq m$, 
       $X_i$ is a Busemann space, $o_i\in X_i$, and $W_i$ is the visual boundary.
 \item For $m< i\leq n$, 
       $X_i$ is a $p_i$-dimensional simply connected solvable Lie group with a lattice, 
       $o_i$ is the unit of $X_i$, and $W_i= S^{{p_i}-1}$ is a corona as 
       in Lemma~\ref{lem:homeo-coarse-map}.
\end{itemize}


We use the following notations.
\begin{align*}
 &Y:= X_1\times \dots \times X_n,\\
 &Y':= \OW_1\times \dots \times \OW_k \times X_{k+1}\times \dots \times X_n,\\
 &Y'':= \OW_1\times \dots \times \OW_k \times X_{k+1}\times \dots \times X_m
  \times \R^{p_{m+1}} \times \dots \times \R^{p_{n}},\\
 &K:= X_1\times \dots \times X_m,\\
 &K':= \OW_1\times \dots \times \OW_k \times X_{k+1}\times \dots \times X_m.
\end{align*}
It is known that $X_i$ is coarsely homotopic to $\OW_i$ for $1\leq i\leq k$. 
See \cite[Section 8]{MR1388312} or \cite[Section 4.7]{WillettThesis}.
Then by Lemma~\ref{lem:product-of-chmtpy}, $Y$ is coarsely homotopic to $Y'$
 and $K$ is coarsely homotopic to $K'$.
By Corollary~\ref{prop:chinvariance-of-chomology}, 
we have 
\begin{align*}
\label{eq:7}
 &KX_*(Y) \cong KX_*(Y'), \,
 K_*(C^*(Y))\cong K_*(C^*(Y')), \, \mu_Y\cong \mu_{Y'}\\
 &KX_*(K) \cong KX_*(K'), \,
 K_*(C^*(K))\cong K_*(C^*(K')),\, \mu_K\cong \mu_{K'}.
\end{align*}
Since homeomorphic coarse maps as in Lemma~\ref{lem:homeo-coarse-map} 
give a homeomorphic coarse map from $Y'$ to $Y''$ and euclidean spaces 
are Busemann spaces, Proposition~\ref{prop:coarsening} and 
Corollary~\ref{cor:LieG-homeo} imply
\begin{align*}
  &KX_*(Y') \cong K_*(Y') \cong K_*(Y'') \cong KX_*(Y''),\\
  &KX_*(K') \cong K_*(K').
\end{align*}
Since all of $\OW_i\, (1\leq i\leq k)$ and $X_j\, (k+1\leq j\leq m)$ 
are scalable, so is $K'$.
Then by Proposition~\ref{prop:scalable-cBC}, the coarse assembly map
\[
 A\colon K_*(K') \to K_*(C^*(K'))
\]
 is an isomorphism.
Combining these isomorphisms, we have the coarse assembly map
\[
 \mu_K\colon KX_*(K) \to K_*(C^*(K))
\]
 is an isomorphism. Then by Proposition~\ref{prop:assembly-map-solv}, 
so is $\mu_{Y}\colon KX_*(Y) \to K_*(C^*(Y))$.

By Theorem~\ref{thm:contractible}, we have the following isomorphism
\[
 K_*(Y'') \to \rK_{*-1}(W_1\star \dots \star W_n).
\]
Combining with the above isomorphisms, we have that the transgression map
\[
 T_{\star W_i}\colon KX_*(Y) \to \rK_{*-1}(W_1\star \dots \star W_n)
\]
is an isomorphism. 
By the fact that $T_{\star W_i} =  b_{\star W_i} \circ \mu_*(Y)$,  
which is proved in \cite[Appendix]{MR1388312}, 
the map $b_{\star W_i}$ is an isomorphism.
\end{proof}

\section{Relatively hyperbolic group}
\label{sec:augmented-space}
In \cite{relhypgrp}, 
the coarse Baum-Connes conjecture for relatively hyperbolic groups was studied.
From the paper we quote notation and facts, 
which we use for the study of products of relatively hyperbolic groups
in the next section. 

Let $G$ be a finitely generated group and $\famP=\{P_1,\dots,P_k\}$ be a
finite family of infinite finitely generated subgroups of $G$ of
infinite index.  We take a finite symmetric generating set $\mathcal{S}$
of $G$ such that $\mathcal{S}_r=\mathcal{S}\cap P_r$ generates $P_r$ for
each $1\le r\le k$.  Also we choose a sequence $\{g_n\}_{n\in\N}$ of $G$
such that $\{g_{(a-1)k+r}\}_{a\in\N}$ is a set of complete
representatives of $G/P_r$ for each $1\le r\le k$.  For $i\in \N$, we
put $(i)= r$ if there exist $a\in \N$ and $1\leq r\le k$ such that
$i=(a-1)k+r$.  We denote by $\Gamma$ the Cayley graph of $G$ with
respect to $\mathcal{S}$ and by $\Gamma_r$ the Cayley graph of $P_r$
with respect to $\mathcal{S}_r$ for each $r$.  We regard $\Gamma_r$ as
the subgraph of $\Gamma$. Then the full subgraph of $\Gamma$ spanned by
$g_iP_{(i)}$ is $g_i\Gamma_{(i)}$ for every $i\in \N$.

Groves and Manning \cite{MR2448064} defined the combinatorial horoball $\Horo(g_iP_{(i)})$
which is a connected graph with the vertices set $g_iP_{(i)}\times \N\cup\{0\}$. 
By identifying the subgraph spanned by $g_iP_{(i)}\times \{0\}$ of $\Horo(g_iP_{(i)})$ with 
the subgraph $g_i\Gamma_{(i)}$ of the Cayley graph $\Gamma$, 
they introduced the augmented space 
\[
\Xaug:= \Gamma \cup \bigcup_{i\in \N} \Horo(g_iP_{(i)}),
\]
which is also a connected graph. 
We equip it with the graph metric. 
For details of the construction, 
see~\cite[Definition 3.12]{MR2448064} and also~\cite[Section 2.1]{relhypgrp}.

Groves and Manning proved that the following is equivalent to 
other definitions of relatively hyperbolic groups.
\begin{definition}[\cite{MR2448064}]
The group $G$ is hyperbolic relative to $\famP$ if $\Xaug$ is 
a Gromov hyperbolic space.
\end{definition}

In this section, we suppose that $G$ is hyperbolic relative to $\famP$, that is, 
$\Xaug$ is a Gromov hyperbolic space.
We gather notation from \cite[Section 2.1 and Notation 5.1]{relhypgrp}. 
\begin{notation}
For each $n\in \N$, we put $X_n:=\Gamma \cup \bigcup_{i\ge n} \Horo(g_iP_{(i)})$. 
In particular $X_1=\Xaug$. Also we put $X_\infty:=\bigcap_{n\ge 1}X_n$. 
Then $X_\infty=\Gamma$. 
For a connected subset $I$ of $[0,\infty)$, 
we denote by $\Horo(g_iP_{(i)}; I)$ the full subgraph of 
$\Horo(g_iP_{(i)})$ spanned by 
$g_iP_{(i)}\times (I\cap \N\cup\{0\})$. 
We put $X(1):=\Gamma \cup \bigcup_{i\in\N} \Horo(g_iP_{(i)};[0,1])$. 
\end{notation}

We remark that all subgraphs of $\Xaug$ on the above are connected.
Hence each of such subgraphs has the graph metric. On the other hand, 
each of them has the restriction of the graph metric of $\Xaug$.
However two metrics are coarsely equivalent and 
thus we do not need to worry about their difference in this paper.  

The following is clear.
\begin{lemma}
\label{coarse-excision}
For each $n\in \N$, 
$X_n= X_{n+1}\cup  \Horo(g_nP_{(n)})$ is a coarse excisive decomposition 
such that $g_n\Gamma_{(n)}= X_{n+1}\cap  \Horo(g_nP_{(n)})$. 
\end{lemma}

In Section 2.3 of \cite{relhypgrp}, we construct the following. 
\begin{enumerate}
\item an anti-\v{C}ech system $\{\U_n\}_{n}$ of $\Xaug$ 
and coarsening maps $\varphi_n\colon \U_n \to \U_{n+1}$;  
\item subsets $\mathcal{X}_n$, $\mathcal{Y}_n$ and $\mathcal{Z}_n$ of $\U_n$.
\end{enumerate}
By the construction, we have the following. See \cite[Proof of Lemma 2.7]{relhypgrp} for (3).
\begin{lemma}\label{preserving}
\begin{enumerate}
\item[(1)] $\abs{\U_n}=\abs{\X_n}\cup \abs{\Y_n}$ and $\abs{\mathcal{Z}_n}=\abs{\X_n}\cap \abs{\Y_n}$.
\item[(2)] $\varphi_n(\abs{\X_n})\subset \abs{\X_{n+1}}$, $\varphi_n(\abs{\Y_n})\subset \abs{\Y_{n+1}}$
and $\varphi_n(\abs{\mathcal{Z}_n})\subset \abs{\mathcal{Z}_{n+1}}$.
\item[(3)] $\abs{\Y_n}\to \abs{\Y_{n+1}}$ factors through $\abs{\Y_n}\times \R_{\ge 0}$
up to proper homotopy, that is, 
there exist proper continuous maps $\abs{\Y_n}\to \abs{\Y_n}\times \R_{\ge 0}$ and 
$\abs{\Y_n}\times \R_{\ge 0}\to \abs{\Y_{n+1}}$ whose composite is properly homotopic to 
$\abs{\Y_n}\to \abs{\Y_{n+1}}$.
\end{enumerate}
\end{lemma}

Now we assume that for $r\in \{1,\dots,k\}$, each $P_r$ admits a finite 
$P_r$-simplicial complex $\EP_r$ which is a universal space for proper actions.
By the Appendix A in \cite{relhypgrp}, there exists a 
finite $G$-simplicial complex $\EG$ which is a universal space
for proper actions such that all $\EP_{r}$ are 
embedded in $\EG$. Moreover we can assume that $G$ is naturally embedded in 
the set of vertices of $\EG$ and $g_iP_{(i)}$ is embedded in $g_i\EP_{(i)}$. 
By identifying $g_i\EP_{(i)}\times \{0\}\subset g_i\EP_{(i)}\times [0,\infty)$
with $g_i\EP_{(i)}\subset \EG$, 
we have a locally compact second countable Hausdorff space
\[
\EX := \EG\cup \bigcup_{i\in \N} (g_i\EP_{(i)}\times [0,\infty)).
\]
See \cite[Section 3]{relhypgrp} for details.
We denote by $\iota$ the natural injection from the vertices set of $\Xaug$ to $\EX$. 
We gather notations from \cite[Section 3 and Notation 5.1]{relhypgrp}.
\begin{notation}
For each $n\in \N$, we put $EX_n:=\EG\cup \bigcup_{i\ge n} (g_i\EP_{(i)}\times [0,\infty))$. 
In particular $EX_1=\EX$. Also we put $EX_\infty:=\bigcap_{n\ge 1}EX_n$. 
Then $EX_\infty=\EG$. 
We put $EX(1) := \EG\cup \bigcup_{i\in \N}  (g_i\EP_{(i)}\times [0,1])$.   
\end{notation}

We equip $\EX$ with the proper coarse structure which is coarsely equivalent to $\Xaug$ by $\iota$
(\cite[Section 6.2]{boundary}).
Then the following is clear.  

\begin{lemma}\label{aaa}
For each $n\in \N$, $\iota$ induces coarse equivalences 
from $\Horo(g_nP_{(n)})$ to $g_n\EP_{(n)}\times [0,\infty)$, 
from $\Horo(g_nP_{(n)},\{1\})$ to $g_n\EP_{(n)}\times \{1\}$, 
from $X_n$ to $EX_n$ and 
from $X(1)$ to $EX(1)$, 
respectively.
\end{lemma}

In Section 3.1 of \cite{relhypgrp}, 
an anti-\v{C}ech system $\{E\U_n\}_{n}$ of $\EX$ in the sense of
\cite[Definition 5.36]{MR2007488} was given. 
We take a proper continuous map $\epsilon: \EX\to E\U_1$ by using a partition of unity. 
Since $\Xaug$ and $\EX$ are coarsely equivalent by $\iota$, 
we know that two sequences 
\[\EX \to \abs{E\U_{1}} \to \abs{E\U_2}\to\cdots\]
\[\abs{\U_{1}} \to \abs{\U_2}\to\cdots\]
are properly homotopic.
Indeed a sequence of proper continuous maps
$\{\phi_n:\abs{E\U_n}\to \abs{\U_{n+1}}\}$ is constructed, 
which implies that two sequences on the above are properly homotopic. 
Then we define $\varphi_0$ as the composite of $\epsilon$ and $\phi_1$.
By the construction, we know the following.

\begin{lemma}
\label{bbb}
After taking a subsequence,
we have the coarsening sequence in a wider sense, 
\[
\EX\to \abs{\U_1}\to \abs{\U_2}\to\cdots
\]  
such that

\begin{enumerate}
\item[(1)] $\varphi_0(EX(1))\subset \abs{\X_1}$,  
$\varphi_0(\bigsqcup_{i\in\N}(g_i\EP_{(i)}\times [1,\infty)))\subset \abs{\Y_1}$ and 
$\varphi_0(\bigsqcup_{i\in\N}(g_i\EP_{(i)}\times \{1\}))\subset \abs{\mathcal{Z}_1}$. 
\item[(2)] $\varphi_0\circ \iota$ is close to a proper continuous map
$\Xaug \to \abs{\U_{1}}$ induced by a partition of unity.
\end{enumerate}
\end{lemma}

The following was essentially proved in \cite{relhypgrp}. 
\begin{lemma}\label{lem:X-tame}
After taking a subsequence, 
$EX(1) \to \abs{\X_1} \to \abs{\X_2}\to\cdots$ 
has a splitting up to proper homotopy. 
\end{lemma}
\begin{proof}
The sequence in the lemma is properly homotopic to 
a coarsening sequence for $X(1)$ by \cite[Proof of Lemma 2.6]{relhypgrp}. 
Since $X(1)$ and $EX(1)$ are coarsely equivalent by Lemma \ref{aaa}, 
the coarsening sequence for $X(1)$
is properly homotopic to any coarsening sequence for $EX(1)$. 
Hence, by the construction of the sequence in the lemma, 
it is strongly properly homotopic to 
any coarsening sequence for $EX(1)$, which has
a splitting up to proper homotopy by Proposition \ref{lem:EG-admits-split}.
Here note that $EX(1)$ and $\EG$ are properly homotopic. 
\end{proof}

By a similar argument, we have the following. 
See \cite[Section 3.1]{relhypgrp}.
\begin{lemma}
\label{lem:Z-tame}
After taking a subsequence, 
$\bigsqcup_{i\in\N}(g_i\EP_{(i)}\times \{1\}) \to \abs{\mathcal{Z}_1} \to \abs{\mathcal{Z}_2}\to\cdots$ 
has a splitting up to proper homotopy. 
\end{lemma}

\section{Proof of Theorem~\ref{th:main-thoerem}}
\label{sec:proof-of-main-th} In Section \ref{sec:proof-setting-step}
we prepare settings for a proof of Theorem~\ref{th:main-thoerem}.  Under
the setting, we prove Theorem~\ref{th:main-thoerem} in Sections
\ref{sec:proof-first-step}, \ref{sec:proof-second-step} and
\ref{sec:proof-final-step}.

\subsection{Settings}
\label{sec:proof-setting-step}
For $1\le i\le l$, let $G_i$ be a hyperbolic group, a CAT(0) group or a polycyclic group.
We take a locally compact Hausdorff space  $\hat{E}G_i$ as follows. 
\begin{itemize}
\item If $G_i$ is a hyperbolic group, then we have a finite 
$G_i$-simplicial complex $\hat{E}G_i$ which is a universal 
space for proper actions \cite{MR1887695}.
\item If $G_i$ is a CAT(0) group. Then 
it properly and cocompactly acts on a proper CAT(0) space by isometries. 
We take such a space $\hat{E}G_i$. 
\item If $G_i$ is a polycyclic group, then 
it is commensurable to a lattice of a simply connected solvable Lie
      group $\LieG_i$.
We put $\hat{E}G_i:= \LieG_i$.
\end{itemize}
We put $\prodG_{[l]}= \prod_{i=1}^l G_i$ and 
$\EGG_{[l]}= \prod_{i=1}^l \hat{E}G_i$. 
Note that $\hat{E}G_i$ is coarsely equivalent to $G_i$.

For $j\in\N$, let $G^j$ be a finitely generated group 
which is hyperbolic relative to $\famP^j:=\{P^j_1,\dots,P^j_{k^j}\}$, where
every $P^j_r$ is an infinite finitely generated subgroup of $G^j$ of infinite index.
In the previous section, for a given group $G$ which is hyperbolic relative to $\famP=\{P_1,\dots,P_{k}\}$,  
we considered 
$\mathcal{S}$, $\Gamma$, $\mathcal{S}_r$ and $\Gamma_r$ for $1\le r\le k$, 
$\{g_i\}_{i\in\N}$, $X_n$ for $n\in \N\cup \{\infty\}$, $X(1)$. 
For each $j\in \N$, we take such sets and denote them by 
$\mathcal{S}^j$, $\Gamma^j$, $\mathcal{S}^j_r$ and $\Gamma^j_r$ for $1\le r\le k^j$, 
$\{g^j_i\}_{i\in\N}$, $X^j_n$ for $n\in \N\cup \{\infty\}$, $X^j(1)$, respectively. 
We put $\mathbb{X}^{[m]}_n= \prod_{j=1}^m X^j_n$ for $m\in \N$.

Now we consider two assumptions.
\begin{enumerate}
\item[(A1)] For every $j\in\N$ and every $1\le r \le k^j$, 
the group $P^j_r$ is a product of some of 
hyperbolic groups, CAT(0)-groups and polycyclic groups.  
\item[(A2)] For every $j\in\N$ and every $1\le r \le k^j$, the group 
$P^j_r$ has a finite $P^j_r$-simplicial complex $\underline{E}P^j_r$ 
which is a universal space of proper $P^j_r$ actions. 
\end{enumerate} 
Assumption (A1) is not used in Section~\ref{sec:proof-first-step}. 
Also Assumption (A2) is not used in Section~\ref{sec:proof-second-step}.

Under Assumption (A2), we can apply the arguments 
in Section~\ref{sec:augmented-space} for each $(G^j,\famP^j)$. 
Then we have $\EP^j_r$ for $1\le r\le k^j$, $\EG^j$, $EX^j_n$ for $n\in \N\cup \{\infty\}$, $EX^j(1)$,
$\iota^j$, $\U^j_n$, $\X^j_n$, $\Y^j_n$, $\mathcal{Z}^j_n$, $\phi^j_n:\U^j_n\to \U^j_{n+1}$ for $n\in \N$, 
$\phi^j_0: EX(G^j,\famP^j)\to \abs{\U^j_1}$, 
which correspond to 
$\EP_r$ for $1\le r\le k$, $\EG$, $EX_n$ for 
$n\in \N\cup \{\infty\}$, $EX(1)$, $\iota$, $\U_n$, $\X_n$, $\Y_n$, 
$\mathcal{Z}_n$, $\phi_n:\U_n\to \U_{n+1}$ for $n\in \N$, 
$\phi_0: \EX\to \abs{\U_1}$ for $(G, \famP)$ 
in the previous section, respectively. 
We put $E\mathbb{X}^{[m]}:= \prod_{j=1}^n EX(G^j,\famP^j)$ and 
$E\mathbb{X}^{[m]}_n = \prod_{j=1}^m EX^j_n$
for $m\in \N$ and $n\in \N$. 
In particular we have $E\mathbb{X}^{[m]}=E\mathbb{X}^{[m]}_1$.

\subsection{Proof of Theorem~\ref{th:main-thoerem}(first step)}
\label{sec:proof-first-step}
In this section we assume (A2). 

Let $Q\to Q_1\to Q_2\to \cdots $ be a sequence of 
locally compact second countable Hausdorff spaces and proper continuous maps 
with a splitting up to proper homotopy.  

\begin{notation}
\begin{align*}
&\abs{\U^{[m]}_n}:= \abs{\U^{1}_n}\times \dots\times \abs{\U^{m}_n}\times Q_n,\\
&\abs{\X^{[m]}_n}:= \abs{\U^{1}_n}\times \dots\times \abs{\U^{m-1}_n}\times \abs{\X^{m}_n} \times Q_n,\\
&\abs{\Y^{[m]}_n}:= \abs{\U^{1}_n}\times \dots\times \abs{\U^{m-1}_n}\times \abs{\Y^{m}_n} \times Q_n,\\
&\abs{\cZ^{[m]}_n}:= \abs{\U^{1}_n}\times \dots\times \abs{\U^{m-1}_n}\times \abs{\cZ^{m}_n} \times Q_n.
\end{align*}
\end{notation}

The coasening sequences
$EX(G^j,\famP^j)\to \abs{\U^j_1}\to \abs{\U^j_2}\to \cdots$
imply a sequence
\[
E\mathbb{X}^{[m]}\times Q\to \abs{\U^{[m]}_1}\to \abs{\U^{[m]}_2}\to \cdots.
\]
This induces the following sequences by Lemmas \ref{preserving} (2)
and \ref{bbb} (1).
\begin{itemize}
\item $E\mathbb{X}^{[m]}\times EX^m(1)\times Q\to \abs{\X^{[m]}_1}\to
\abs{\X^{[m]}_2}\to \cdots$.
\item $E\mathbb{X}^{[m]}\times \bigsqcup_i(g^m_i\EP^m_{(i)}\times
[1,\infty))\times Q\to \abs{\Y^{[m]}_1}\to \abs{\Y^{[m]}_2}\to
\cdots$.
\item $E\mathbb{X}^{[m]}\times \bigsqcup_i(g^m_i\EP^m_{(i)}\times
\{1\})\times Q\to \abs{\mathcal Z^{[m]}_1}\to \abs{\mathcal
Z^{[m]}_2}\to \cdots$.
\end{itemize}

\begin{lemma}
\label{null}
The map $\abs{\Y^{[m]}_n} \to \abs{\Y^{[m]}_{n+1}}$ induces a null map in the $K$-homology. 
\end{lemma}
\begin{proof}
Lemma \ref{preserving} (3) implies that 
$\abs{\Y^{[m]}_n} \to \abs{\Y^{[m]}_{n+1}}$ factors through 
$\abs{\Y^{[m]}_n}\times \R_{\ge 0}$ up to proper homotopy. 
The lemma follows from the fact that $K_*(\abs{\Y^{[m]}_n}\times \R_{\ge 0})=0$. 
\end{proof}

The following is the key to the proof of Theorem~\ref{th:main-thoerem}.
\begin{proposition}
\label{prop:some-aug-weak-coarsening}
We have
\[
K_*(E\mathbb{X}^{[m]}\times Q) \cong \Im(K_*(\abs{\U^{[m]}_n})\to K_*(\abs{\U^{[m]}_{n+1}}))
\cong \varinjlim K_*(\abs{\U^{[m]}_n}). 
\]
Especially if $Q\to Q_1\to Q_2\to $ is a coarsening sequence for $Q$ 
in a wider sense
with a splitting up to proper homotopy, 
then $K_*(E\mathbb{X}^{[m]}\times Q) \cong KX_*(E\mathbb{X}^{[m]}\times Q)$.
\end{proposition}

We prove the proposition by induction on $m$. 
We fix $m\ge 0$. Now we assume that for 
any sequence $Q\to Q_1\to Q_2\to \cdots$ of 
locally compact second countable Hausdorff spaces and proper continuous maps
with a splitting up to proper homotopy, the following isomorphism holds.
\begin{align}
\label{eq:assumption}
K_*(E\mathbb{X}^{[m]}\times Q) \cong \Im(K_*(\abs{\U^{[m]}_n})\to K_*(\abs{\U^{[m]}_{n+1}})). 
\end{align}
Here as the case where $m=0$, we considered that 
$K_*(Q) \cong \Im(K_*(Q_n)\to K_*(Q_{n+1}))$, which is clear.

\begin{lemma}
\label{eq:8-9}
Under the assumption of~(\ref{eq:assumption}), we have
\begin{align}
\label{eq:8}
 &K_*(E\mathbb{X}^{[m]}\times EX^{m+1}(1) \times Q) 
 \cong \Im(K_*(\abs{\X^{[m+1]}_n})\to K_*(\abs{\X^{[m+1]}_{n+1}})),\\
\label{eq:9}
 &K_*(E\mathbb{X}^{[m]}\times 
 \bigsqcup_i \left(g^{m+1}_i\EP^{m+1}_{(i)}\times \{1\} \times Q\right)) 
 \cong \Im K_*(\abs{\cZ^{[m+1]}_n})\to K_*(\abs{\cZ^{[m+1]}_{n+1}})).
\end{align}
\end{lemma}

\begin{proof}
By Lemma~\ref{lem:X-tame}, a sequence 
$EX^{m+1}(1)\to \abs{\X^{m+1}_1}\to \abs{\X^{m+1}_2}\to \cdots$ 
has a splitting up to proper homotopy. 
Thus 
$EX^{m+1}(1)\times Q\to \abs{\X^{m+1}_1} \times Q_1\to \abs{\X^{m+1}_n} \times Q_2\to \cdots$ 
has a splitting  up to proper homotopy.
Then (\ref{eq:8}) follows from the induction hypothesis~(\ref{eq:assumption}).
We can also apply Lemma~\ref{lem:Z-tame} to show (\ref{eq:9}) by the same argument.
\end{proof}

Now we consider the Mayer-Vietoris sequences for 
\begin{align*}
\abs{\U^{[m]}_n} &= \abs{\X^{[m]}_n}\cup \abs{\Y^{[m]}_n},\\
\abs{\U^{[m]}_{n+1}} &= \abs{\X^{[m]}_{n+1}}\cup \abs{\Y^{[m]}_{n+1}}, \\
E\mathbb{X}^{[m+1]}\times Q &= E\mathbb{X}^{[m]} \times EX^{m+1}(1) \times Q
\cup E\mathbb{X}^{[m]} 
\times \bigsqcup_i \left(g^{m+1}_i\EP^{m+1}_{(i)} \times [1,\infty) \right)\times Q. 
\end{align*}
Note that 
$\abs{\mathcal{Z}^{[m]}_n}= \abs{\X^{[m]}_n}\cap \abs{\Y^{[m]}_n}$ and 
\begin{align*}
 &E\mathbb{X}^{[m]} \times 
 \bigsqcup_i \left(g^{m+1}_i\EP^{m+1}_{(i)} \times \{1\} \right)\times Q =\\
 &E\mathbb{X}^{[m]} \times EX^{m+1}(1) \times Q
 \cap E\mathbb{X}^{[m]} \times 
 \bigsqcup_i \left(g^{m+1}_i\EP^{m+1}_{(i)} \times [1,\infty) \right)\times Q
\end{align*}
Also we remark that the $K$-homology of 
$E\mathbb{X}^{[m]} \times \left(\bigsqcup_i g^{m+1}_i\EP^{m+1}_{(i)} \times [1,\infty) \right)\times Q$
is trivial because it is homeomorphic to 
$E\mathbb{X}^{[m]} \times \left(\bigsqcup_i g^{m+1}_i\EP^{m+1}_{(i)}\right)\times Q \times \R_{\ge 0} $. 
By a diagram chasing with Lemmas \ref{null} and \ref{eq:8-9}, we have
\begin{align}
K_*(E\mathbb{X}^{[m+1]}\times Q) \cong 
\Im(K_*(\abs{\U^{[m+1]}_n})\to K_*(\abs{\U^{[m+1]}_{n+1}})). 
\end{align}
This finishes the proof of Proposition~\ref{prop:some-aug-weak-coarsening}.

By Proposition~\ref{lem:EG-admits-split} and Proposition~\ref{lem:example-split-anti-Cech-sys},
the space $\EGG_{[l]}$ admits a coarsening sequence with a splitting up to proper homotopy. 
By Proposition~\ref{prop:some-aug-weak-coarsening}, Lemma \ref{aaa}, 
Lemma \ref{bbb} (2) and Mayer-Vietoris arguments, 
we have the following.
\begin{corollary}
\label{cor:coarsening-product-Q}
For all $n\in \N$, we have
\begin{align}
\label{eq:10}
K_*(E\mathbb{X}^{[m]}_n \times \EGG_{[l]}) \cong KX_*(\mathbb{X}^{[m]}_n \times \prodG_{[l]}).
\end{align}
\end{corollary}

\subsection{Proof of Theorem~\ref{th:main-thoerem} (second step)}
\label{sec:proof-second-step}
In this section we prove the following under (A1) by using Theorem~\ref{th:corona-assembly}. 
\begin{proposition}
\label{lem:Xn} Let $\mathbb{Y}$ be a product of finitely many proper
geodesic Gromov hyperbolic spaces and $\prodG$ be a finite product of some of 
hyperbolic groups, CAT(0) groups and polycyclic groups.  Let
$(n_1,\dots,n_m) \in \N^m$.  Set $X_{(n_1,\dots,n_m)} = X^1_{n_1}\times
\dots \times X^j_{n_m}$.  Then the coarse assembly map
\[
\mu_*(X_{(n_1,\dots,n_m)}\times \mathbb{Y} \times \prodG) 
\colon KX_*(X_{(n_1,\dots,n_m)}\times \mathbb{Y} \times \prodG) 
\to K_*(C^*(X_{(n_1,\dots,n_m)}\times \mathbb{Y} \times \prodG))
\]
is an isomorphism. In particular the coarse assembly map 
\[
\mu_*(\mathbb{X}^{[m]}_n \times \prodG_{[l]}) 
\colon KX_*(\mathbb{X}^{[m]}_n \times \prodG_{[l]}) 
\to K_*(C^*(\mathbb{X}^{[m]}_n \times \prodG_{[l]}))
\]
is an isomorphism for each $n\in \N$.
\end{proposition}

The following lemma is the first step of induction in the proof of Proposition~\ref{lem:Xn}.
\begin{lemma}
\label{lem:A}
Let $\mathbb{Y}$ and $\prodG$ be as in Lemma \ref{sec:proof-second-step}. 
For all $n\geq 1$,
\[
\mu_*(X^1_n \times \mathbb{Y} \times \prodG) \colon 
KX_*(X^1_n \times \mathbb{Y} \times \prodG) 
\to K_*(C^*(X^j_n \times \mathbb{Y} \times \prodG)).
\]
is an isomorphism.
\end{lemma}
\begin{proof}
We prove the assertion by induction on $n$. Since $X^1_1$ is Gromov
hyperbolic, $\mu_*(X^1_1 \times \mathbb{Y} \times \prodG)$ is an
isomorphism by Theorem~\ref{th:corona-assembly}.  Now we consider the
Mayer-Vietoris exact sequence for the coarse excisive decomposition (see
Lemma \ref{coarse-excision})
\begin{align*}
X^1_n \times \mathbb{Y} \times \prodG = X^1_{n+1} \times \mathbb{Y} \times \prodG 
\cup \Horo(g^1_nP^1_{(n)}) \times \mathbb{Y} \times \prodG,
\end{align*}
where the intersection is
\begin{align*}
g^1_nP^1_{(n)} \times \mathbb{Y} \times \prodG  = 
X^1_{n+1} \times \mathbb{Y} \times \prodG 
\cap \Horo(g^1_nP^1_{(n)}) \times \mathbb{Y} \times \prodG. 
\end{align*}
By (A1) and Theorem~\ref{th:corona-assembly}, the coarse assembly map
$\mu_*(g^1_nP^1_{(n)} \times \mathbb{Y} \times \prodG)$ is an isomorphism. Thus, if
$\mu_*(X^1_n \times \mathbb{Y} \times \prodG)$ is an isomorphism, so is 
$\mu_*(X^1_{n+1} \times \mathbb{Y} \times \prodG)$.
\end{proof}

\begin{proof}[Proof of Proposition~\ref{lem:Xn}]
We prove the assertion by induction on $m$. 
By Lemma~\ref{lem:A} we know the case where $m=1$. 

Fix $m\in \N$. 
Now we suppose for all $(n_1,\dots,n_m)\in \N^m$, and for all
$\mathbb{Y}, \prodG$ satisfying the condition in Proposition~\ref{lem:Xn},
the map 
$\mu_*(X_{(n_1,\dots,n_m)}\times \mathbb{Y} \times \prodG)$ is an isomorphism.
Since $X^{m+1}_1$ is a Gromov hyperbolic space, 
$X^{m+1}_1\times \mathbb{Y}$ is a product of finitely many proper geodesic 
Gromov hyperbolic spaces. Thus by the assumption, the map 
\[
\mu_*(X_{(n_1,\dots,n_m,1)} \times \mathbb{Y} \times \prodG )
= \mu_*(X_{(n_1,\dots,n_m)} \times X^{m+1}_1\times \mathbb{Y} \times \prodG )
\]
is an isomorphism. 
Now we fix $n_{m+1}\in \N$ and assume that 
\[
 \mu_*(X_{(n_1,\dots,n_m,n_{m+1})} \times \mathbb{Y} \times \prodG )
\]
is an isomorphism.
We consider the coarse excisive decomposition (see Lemma \ref{coarse-excision})
\[
X_{(n_1,\dots,n_m,n_{m+1})} \times \mathbb{Y} \times \prodG 
=  X_{(n_1,\dots,n_m,n_{m+1}+1)} \times \mathbb{Y} \times \prodG 
\cup X_{(n_1\dots,n_m)}\times \Horo(g^{m+1}_{n_{m+1}}P^{m+1}_{(n_{m+1})}) 
\times \mathbb{Y} \times \prodG.
\]
The intersection is 
$X_{(n_1,\dots,n_m)} \times g^{m+1}_{n_{m+1}}\Gamma^{m+1}_{(n_{m+1})} \times \mathbb{Y} \times \prodG$. 
Set $\prodG':= P^{m+1}_{(n_{m+1})} \times \prodG$, which is coarsely equivalent to 
$g^{m+1}_{n_{m+1}}\Gamma^{m+1}_{(n_{m+1})}\times \prodG$.
By Assumption (A1), the group $\prodG'$ satisfies the condition of 
 Proposition~\ref{lem:Xn}. By the induction hypothesis,
$\mu_*(X_{(n_1,\dots,n_m)} \times \mathbb{Y} \times \prodG')$ and thus 
$\mu_*(X_{(n_1,\dots,n_m)} \times g^{m+1}_{n_{m+1}}\Gamma^{m+1}_{(n_{m+1})} \times \mathbb{Y} \times \prodG)$ 
are isomorphisms. 
Then, by the Mayer-Vietoris sequence, the map 
$\mu_*(X_{(n_1,\dots,n_m,n_{m+1}+1)} \times \mathbb{Y} \times \prodG )$ is an isomorphism. 
\end{proof}

\subsection{Proof of Theorem~\ref{th:main-thoerem} (final step)}
\label{sec:proof-final-step}
Now we will finish the proof of Theorem~\ref{th:main-thoerem}.
Under (A1) and (A2), we have 
that 
\[ 
K_*(E\mathbb{X}^{[m]}_n \times \EGG_{[l]})
\cong KX_*(\mathbb{X}^{[m]}_n \times \prodG_{[l]}) 
\cong K_*(C^*(\mathbb{X}^{[m]}_n \times \prodG_{[l]})). 
\]
for all $n\in \N$ 
by Corollary~\ref{cor:coarsening-product-Q} and Proposition \ref{lem:Xn}.
Also since $\prod_{j=1}^m \EG^j \times \EGG_{[l]}$ is coarsely equivalent to 
$\prod_{j=1}^m G^j \times \prodG_{[l]}$, 
Proposition~\ref{lem:EG-admits-split} and Proposition~\ref{lem:example-split-anti-Cech-sys} imply
\[ 
K_*(\prod_{j=1}^m \EG^j \times \EGG_{[l]})
\cong KX_*(\prod_{j=1}^m G^j \times \prodG_{[l]}). 
\]
Hence it is sufficient to show that 
\[
K_*(E\mathbb{X}^{[m]}_\infty \times \EGG_{[l]})
\cong K_*(C^*(\mathbb{X}^{[m]}_\infty \times \prodG_{[l]})).
\]
Here note that $E\mathbb{X}^{[m]}_\infty=\prod_{j=1}^m \EG^j$ and 
$\mathbb{X}^{[m]}_\infty=\prod_{j=1}^m \Gamma^j$. 

By the same argument as the one in \cite[Section 5.2.]{relhypgrp},
we have the following commutative diagram such that upper and lower
horizontal sequences are exact:
\begin{eqnarray*}
 \xymatrix{
   0  \ar[r]& \limone K_{p+1}(E\mathbb{X}^{[m]}_n \times \EGG_{[l]}) \ar[d] \ar[r]
    & K_p(E\mathbb{X}^{[m]}_\infty \times \EGG_{[l]}) \ar[d]\ar[r] 
    & \varprojlim K_p(E\mathbb{X}^{[m]}_n \times \EGG_{[l]}) \ar[d] \ar[r]
    & 0 .\\
  0 \ar[r]& \limone K_{p+1}(C^*(\mathbb{X}^{[m]}_n \times \prodG_{[l]})) \ar[r]& 
   K_p(C^*(\mathbb{X}^{[m]}_\infty \times \prodG_{[l]})) \ar[r]
    & \varprojlim K_p(C^*(\mathbb{X}^{[m]}_n \times \prodG_{[l]})) \ar[r] & 0.
}
\end{eqnarray*}
By the five lemma, we have the desired isomorphism.


\bibliographystyle{amsplain}
\bibliography{/Users/tomo/Library/tex/math}

\address{ Tomohiro Fukaya \endgraf
Mathematical institute, Tohoku University, Sendai 980-8578, Japan}

\textit{E-mail address}: \texttt{tomo@math.tohoku.ac.jp}

\address{ Shin-ichi Oguni\endgraf
Department of Mathematics, Faculty of Science,
Ehime University,
2-5 Bunkyo-cho,
Matsuyama,
Ehime,
790-8577 Japan
}

\textit{E-mail address}: \texttt{oguni@math.sci.ehime-u.ac.jp}

\end{document}